\documentclass[preprint,11pt,3p,review]{elsarticle}
\usepackage{amssymb}
\usepackage{amsmath}
\usepackage[ruled,vlined]{algorithm2e}
\usepackage{enumerate}
\usepackage{enumitem}
\usepackage{tabularx,booktabs,mathrsfs,setspace}%
\usepackage[dutch,english]{babel}
\usepackage{soul}
\usepackage{mathtools}
\usepackage{enumerate}
\usepackage{nicefrac}
\usepackage{caption}
\usepackage{subcaption}
\usepackage{natbib}
 \bibpunct[, ]{(}{)}{,}{a}{}{,}%
 \def\bibsep{\smallskipamount}%
 \usepackage{breqn}
\usepackage{epsfig}
\usepackage{bigstrut}
\usepackage{multirow}
\usepackage{dsfont}
\usepackage{makecell}
\usepackage{enumitem}
\setlist{nolistsep}
\usepackage[utf8]{inputenc}
\usepackage[english]{babel}
\usepackage{amsmath}
\usepackage{amssymb}
\usepackage{amsthm}
\usepackage{float}

\newtheorem{lemma}{Lemma}
\usepackage{array,multirow,graphicx}
\usepackage{tikz}
\usetikzlibrary{arrows}
\usetikzlibrary{shapes.geometric}
\usetikzlibrary{positioning}
\usetikzlibrary{calc}
\usetikzlibrary{fit}
\usetikzlibrary{arrows.meta}
\usetikzlibrary{decorations.pathreplacing}
\usetikzlibrary{quotes}
\usepackage{graphicx}
\usepackage{multirow}
\def\N{\mathbb{N}}
\def\R{\mathbb{R}}

\def \P {\mathbb{P}}

\def \-> {\rightarrow}

\journal{}

\begin{document}

\begin{frontmatter}
\title{Dedicated maintenance and repair shop control for spare parts networks}

\author[label1]{Chaaben Kouki\corref{cor1}}
\address[label1]{Essca School of Management, Department of Operations Management and Decision Sciences, Angers, France}

\cortext[cor1]{Corresponding author}
\ead{chaaben.kouki@essca.fr}

\author[label2]{Melvin Drent}
\ead{m.drent@tue.nl}

\author[label2]{Collin Drent}
\ead{c.drent@tue.nl}

\author[label3]{M. Zied Babai}
\ead{mohamed-zied.babai@kedgebs.com}

\address[label2]{Eindhoven University of Technology, Department of Industrial Engineering and Innovation Sciences, Eindhoven, the Netherlands}

\address[label3]{Kedge Business School, Bordeaux, France}

\begin{abstract}
We study a repairable inventory system dedicated to a single component that is critical in operating a capital good. The system consists of a stock point containing spare components, and a dedicated repair shop responsible for repairing damaged components. Components are replaced using an age-replacement strategy, which sends components to the repair shop either preventively if it reaches the age-threshold, and correctively otherwise. 
Damaged components are replaced by new ones if there are spare components available, otherwise the capital good is inoperable. If there is free capacity in the repair shop, then the repair of the damaged component immediately starts, otherwise it is queued. The manager decides on the number of repairables in the system, the age-threshold, and the capacity of the repair shop. 
There is an inherent trade-off: A low (high) age-threshold reduces (increases) the probability of a corrective replacement but increases (decreases) the demand for repair capacity, and a high (low) number of repairables in the system leads to higher (lower) holding costs, but decreases (increases) the probability of downtime. We first show that the single capital good setting can be modelled as a closed queuing network with finite population, which we show is equivalent to a single queue with fixed capacity and state-dependent arrivals. For this queue, we derive closed-form expressions for the steady-state distribution. We subsequently use these results to approximate performance measures for the setting with multiple capital goods.
\end{abstract}


\end{frontmatter}

% \linenumbers

\section{Introduction}
Spare parts are becoming ubiquitous in modern societies. They are used in after-sales services and in maintaining operating systems and capital goods within most companies, which renders their management an important lever to improve the service to customers and to reduce costs \citep{al2013interval, sleptchenko2018joint, topan2020review}. The management of repairable spare parts inventory systems is challenging since it consists not only in controlling the inventories used to satisfy demand triggered by planned and unplanned maintenance activities but also in managing the repair shops that repair the \textcolor{black}{spare} parts often with a limited capacity and replenish the inventory points \citep{tiemessen2013reducing, arts2016repairable, drent2021expediting}. The management of such systems require integrated design and control decisions. The design decisions often relate to  the capacity of the repair shops (e.g., number of repairmen), whereas the control decisions consist in scheduling the production within the repair shops and managing the inventory replenishment in the stock points. Examples of companies that typically make integrated inventory and repair shop decisions for spare parts are service companies such as transportation companies (e.g., trains, subways, and airlines) and military organizations (e.g., frigates and cruisers). 
The planning of the overhauls or the age replacement policy of the various parts is generally fixed when capital assets start servicing and taking into account the age of the parts and their probability of failure since the last replacement/failure.

The modeling of repairable spare parts systems is very complex and depends largely on the typology of the repair shop and the allocation of its capacity repairmen as well as the availability of the required parts for maintenance activities. The literature dealing with these systems is mainly built on a research body using the METRIC (Multi-Echelon Technique for Recoverable Item Control) model that extends since the 1960s \citep{sherbrooke1968metric,  muckstadt2010principles}. However, it is important to notice that most of this research has looked at systems triggered by corrective maintenance activities, and research papers that consider both corrective and preventive maintenance are lacking, especially when the repair centers have a limited capacity. Moreover, this literature typically assumes that failures of capital goods follow an exponential distribution, i.e., the spare parts inventory consumption is given by Poisson demand processes. The reality is that the exponential distribution admits a constant failure rate that does not represent the actual failure rate of a capital good. The use of the exponential distribution is probably due to the fact that the failure rate of several capital goods together could be approximated by an exponential distribution, but in no case does the failure of a single capital good admit an exponential distribution, i.e., a Poisson process for the part's demand \citep{saidane2013performance}. Note that empirical goodness-of-fit experiments conducted for spare parts inventories show that spare parts demand often provides a low fit to the Poisson process \citep{lengu2014spare, turrini2019spare}. Our aim in this paper is to fill the research gaps in the literature presented above.

In this paper, we consider a system with one capital good and one repair shop having a limited capacity. We assume that unplanned maintenance follows a general distribution. Moreover, we include planned maintenance by setting an age for each part installed in the capital good at which the part should be replaced by a new one from the stock point (if any). Our aim is to determine the required inventory level and the optimal repair shop capacity as well as the optimal age-threshold so that the total cost of the system is minimized. For this system, we perform an exact analysis and we provide the optimal stock level as well as the optimal capacity of the repair shop and the optimal threshold age at which planned maintenance should be performed. We then extend the work by studying a system with multiple capital goods and a repair shop with limited capacity. We first consider the case that the repair shop does not differentiate between capital goods and performs the repairs on a first-come, first-served basis. We then allow for the setting where the repair shop differentiates between the capital goods according to their criticality and repairs failed parts from these capital goods according a certain priority rule. For this system, it is difficult to provide an exact analysis due to the fact that this network of capital goods and the common repair shop do not admit a simple steady-state probability of having $n$ spare parts available for each capital good in the network, unlike Jacksonian network systems. We therefore analyze this system approximately using a decomposition algorithm that is based on our exact results obtained in the case a single capital good. This approximate method appears to be quite accurate. Hence, the contribution of this paper is twofold.
\begin{enumerate}
    \item This is the first work that analyses a repairable spare parts system composed of one repair shop with limited capacity and one stock point facing a generally distributed demand. The cases of one capital good and multiple capital goods are considered (an exact analysis is developed in the case of one capital good and an approximate one in the case of multiple capital goods). We propose an algorithm that enables to optimize the base stock level, the number of repairables in the system, the age-threshold, and the capacity of the repair shop. 
\item We consider both corrective and preventive maintenance strategies in a spare parts inventory system with a capacitated repair center. To the best of our knowledge, this is the first paper that makes an integration of both maintenance strategies within the METRIC modelling framework in addition to considering a capacitated repair.
 \end{enumerate}

The remainder of this paper is organised as follows. In Section \ref{sec:LitRev}, we provide an overview of the literature related to spare parts inventory control with a focus on the research dealing with repairable inventory systems. Section \ref{singleCapital} is dedicated to the system description ad the analysis in the case of a single capital good. The analysis is extended in Section \ref{multipleCapital} to the case of a system with multiple capital goods. We present in Section \ref{sec:numericalResults} the results of the numerical investigation. The conclusions of our work are provided in Section \ref{sec:concl} along with some avenues for further research.

%In both cases, the lifetime of the capital goods is assumed to be a general distribution. We set different age-thresholds at which components must be removed for inspection, reconditioning or reparation. 

%It should be noted that the model we consider aims at optimizing the age-threshold at which preventive maintenance must be performed.  There are many cases where preventive maintenance is based on a fixed schedule predetermined by the manufacturer. This does not exclude inspections and preventive maintenance from being carried out according to the conditions of use and the degradation undergone by the assets. 

\section{Literature review} \label{sec:LitRev}
This paper integrates preventive maintenance with spare parts control. As such, it contributes to two streams of literature. The first stream is on repairable spare parts inventory systems while the second one focuses on preventive maintenance.  Below we will describe those papers in these fields and its intersection most related to our research. For a general overview of spare part inventory control, we refer the reader to \cite{kennedy2002overview}, \cite{porras2008inventory}, \cite{lengu2014spare} and \cite{hu2018or}. Likewise, general overviews on maintenance optimization can be found in \cite{dekker1996applications} and \cite{de2019review}.

There is an extensive literature on inventory systems of repairable spare parts. We recall that this literature often considers multi-echelon systems that include a repair shop that replenishes one or several stock points of multiple components dedicated to maintaining a capital good or a set of external customers. The major challenge of this research is to minimize the inventory holding cost at the stocking points and the operating cost of the repair facilities as well as providing fast recovery service to the capital good or customers \citep{tiemessen2013reducing}. This stream of literature can be roughly split into (i) a research area where the stock point is replenished by the repair shop as a single supply source, and (ii) a research body where the stock points are replenished by a dual supply, i.e. an expediting source in addition to the repair shop sourcing.

The earliest contributions on repairable spare parts inventory focus on the METRIC model proposed by \citep{sherbrooke1968metric} and its variants such as the MOD-METRIC model \citep{muckstadt1973model} and the VARI-METRIC model \citep{slay1984vari}. The basic METRIC model consists of a spare parts inventory system consisting of multiple stock points facing demand for spare parts, which are each replenished by a central repair center. The failures of the parts at the stock points occur according to Poisson processes, so that the distribution of parts in the repair center is again Poisson distributed, independent of the service time distribution. Moreover, both the stock points and the repair center are controlled by a base-stock policy. A collection of the research findings related to the METRIC-related research is given in the books by \cite{muckstadt2004analysis} and \cite{muckstadt2010principles}.

It is worth noting that most of the works in the literature on METRIC models assumes ample (infinite) capacity of the repair facility and an infinite parts population in the whole system. Some exceptions that deal -- like the work in this paper -- with capacitated systems can be found in \cite{gross1978spares}, \cite{gross1983closed}, \cite{graves1985multi}, \cite{albright1993steady} and \cite{diaz1997models}. With regard to the literature on capacitated systems, most works rely on closed queueing network models, thereby often being restricted to failure rates given by Poisson or compound-Poisson processes. \cite{graves1985multi} showed the benefit of considering a finite capacity and assuming a compound-Poisson demand process, such that demand processes with varying mean-to-variance rates can be better estimated. This was shown to significantly outperform the basic METRIC model due to a better estimation of the spare parts requirements in the system under a high repair shop utilization.  \cite{diaz1997models} developed extensions where, in addition to the compound Poisson demand process assumption, general repair distributions are considered and the case of different classes of parts in the system, each with a different repair distribution. \cite{zijm2003capacitated} considered a two-indenture maintenance system composed of a capacitated repair shop and a number of identical installations that are considered as assemblies made up of a number of repairable components. They modeled the repair shop and the assembly facility as product form queuing networks, where the component failures of the assemblies occur according to a Poisson process and the the stock points of the assemblies are controlled with a base-stock policy. They developed accurate approximation procedures to evaluate the performance of the system by means of the fill rate and the work in process level. More recently, \cite{park2011multi, park2014approximation} analysed the case of an inventory system composed of multiple stock points that order parts from a central warehouse, which is controlled with a continuous review $(S,Q)$ inventory policy. The failure of the parts follows a Poisson process and the repair time in the repair shop that replenishes the depot follows either an exponential or a two-phase Coxian time distribution. They provided expressions for the steady-state probabilities of the system and proposed an approximation method to estimate the parameters that minimise the costs with a good computational effort. In our paper, we too consider a capacitated system where the repair shop is controlled according to a base-stock policy, but contrary too all the works above, we assume a general distribution for the spare parts demand at the stock points. We also propose an extension where the inventory system is composed of multiple types of capital goods.

With regard to the literature that considers that the stock point can be supplied by a dual sourcing including a repair expediting, we refer to the research by \cite{arts2016repairable} and \cite{drent2021expediting}. This research builds on the early work on dual source inventory systems developed by \cite{moinzadeh1991s}, \cite{song2009inventories} among others. A review on the inventory control of such multiple supply systems is provided by \cite{svoboda2021typology}. \cite{arts2016repairable} analyse a dual sourcing spare parts inventory system with an expedited sourcing. They make the assumption of a Markov modulated Poisson process and relax the assumptions that the demand process is a stationary Poisson process (often considered in the literature). 

It should be noted that all the above described literature assumes, to keep their analyses based on queueing networks tractable, that the spare parts demand and the replenishment of the stock are only based on the failure of the part, which means that it is under a corrective maintenance strategy. The literature does not consider that spare parts have limited lifetimes and can potentially be replaced through a preventive maintenance -- before they fail -- when the lifetime reaches an age-threshold \citep{poppe2017numerical,drent2020censored}. This gap in the literature constitutes one of the main contributions of this paper where we assume that in addition to corrective maintenance due to a failure that generates the replacement of the parts, there is a preventive maintenance age-threshold that triggers the replenishment (i.e. the demand) of parts at the stock point. The age-threshold has an impact on the arrival process at the repair shop, and should thus be taken into account in evaluating the performance of the system as a whole.     

Finally, it is worth mentioning that the idea of integrating corrective and preventive maintenance strategies in spare parts inventory systems is not new. In fact, there is a body of literature that deals with maintenance strategies and the management of spare parts characterized with a limited lifetime \citep[see, e.g.,][]{keizer2017joint,poppe2017numerical, westerweel2019preventive,castro2019opportunistic,salari2020joint}. \cite{poppe2017numerical}, among others, showed that compared to corrective maintenance, preventive maintenance policies increase the demand for spare parts, as some remaining useful lifetime of components may not be used, which increases the inventory and capacity requirements. However, there is an agreement in the literature that the cost of the corrective maintenance, especially if the parts are not available in stock (downtime cost) may be much higher \citep{eruguz2018integrated}. This cost is likely to increase even more if the parts are replenished from a capacitated repair center. This latter effect has not been investigated in the literature because a common assumption is that spare parts are ordered from an outside supplier with ample capacity.  To the best of our knowledge, integrating maintenance strategies when minimizing the inventory costs and the operating costs (the maintenance costs of the capital good and the operating costs at the repair center) in an inventory system of repairable parts has not been investigated in the literature. In this paper, we study this intriguing relationship and fill this gap in the literature.

\section{System description, modelling and analysis} \label{singleCapital}
In this section, we consider the case of one capital good. We describe the system, we derive the steady state probabilities and we present the cost function and its properties. The analysis in the case of a system with multiple capital goods is presented in Section \ref{multipleCapital}. 

\subsection{System description and assumptions}

We consider a repairable inventory system consisting of a stock point responsible for maintaining a single capital good and a repair shop responsible for repairing components. The capital good consists of a number of critical components that fail infrequently and independently. These critical components are crucial for operating the capital good, i.e. the capital good is down if one of these components fails. The components are at such levels in the material breakdown structure of the capital good that they can be replaced as a whole by spare parts. The random lifetime of a critical component, denoted with $X$, is generally distributed with distribution $\P(X\leq x) = F(x)$, $f(x) = \frac{\mathrm{d} F(x)}{\mathrm{d} x}$.

All components are controlled by the canonical age-based maintenance policy introduced by \cite{barlow1960optimum}. That is, components are replaced when their lifetime reaches an age-threshold  $\tau > 0$, or upon failure, whichever comes first.  Components that fail before reaching $\tau$ have severe consequences because it is not planned (e.g. think of an assembly line that is stopped abruptly, trains that stand still, or an airplane crash). On the other hand, components whose ages reach $\tau$ do not have such severe consequences as their maintenance practices can be planned for.

Upon replacement of a component (either preventively or correctively), it is immediately replaced by a new spare part from the stock point. If the required spare part is not on hand, then the capital good remains down until a spare part is available again. We assume that replaced parts can always be repaired (no condemnation) and that these parts are sent to the repair shop immediately, where, if none of the $K$ repair stations are available, joins a queue of parts awaiting repair. 
The repair time is random, and we make the common assumption \cite[see, e.g.,][]{tian2021optimal,li2022after} that this time is exponentially distributed with rate $\mu_r>0$. After the replaced part is repaired it is immediately shipped back to the stock point, where it is then available again. The size of the turn-around stock in our system, $S\in\mathbb{N}_0$, of the repairables is determined at time $t=0$ and cannot be adapted afterwards. Figure \ref{fig:model1} illustrates that this repairable inventory system -- for a single capital good -- can be viewed as a cyclic tandem queue network with a fixed number of customers equal to base stock level $S$.

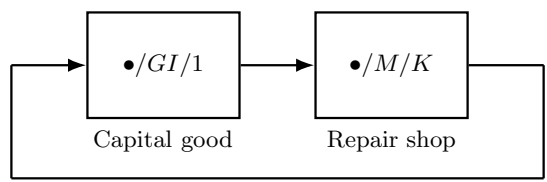
\begin{figure}[ht!]
    \centering
\fontsize{9pt}{12pt}\selectfont
\begin{tikzpicture}[
    triangle/.style = {draw, line width = 0.3mm, regular polygon, regular polygon sides=3, minimum size=2.2cm},
    node rotated/.style = {rotate=180},
    border rotated/.style = {shape border rotate=180},
  ]
	\node (rect2) at (2,1.5) [draw,line width = 0.3mm,minimum width=2cm,minimum height=1.4cm] {$\bullet/GI/1$};
	\node (rect1) at (5,1.5) [draw,line width = 0.3mm,minimum width=2cm,minimum height=1.4cm] {$\bullet/M/K$};
	\node [text width=3cm, align=center, inner sep=0pt] at (2,0.5) {Capital good};
	\node [text width=3cm, align=center, inner sep=0pt] at (5,0.5) {Repair shop};
	\draw [-Latex, line width = 0.3mm] (0,1.5) -- (rect2.west);
	\draw [-Latex, line width = 0.3mm] (rect2.east) -- (rect1.west);
	\filldraw [ line width = 0.3mm] (rect1.east) -- (7,1.5);
	\draw [ line width = 0.3mm] (7,1.5) to (7,0);
	\draw [ line width = 0.3mm] (7,0) -- (0,0);
	\draw [ line width = 0.3mm] (0,0) -- (0,1.5);
\end{tikzpicture}
    \caption{Cyclic tandem queue with a fixed number of customers}
    \label{fig:model1}
\end{figure}

To see this, let $n_o\in\mathbb{N}_0$ and $n_r\in\mathbb{N}_0$, with $n_o+n_r =S$, denote the number of spare parts in stock, including one part is in operation,  and in the repair shop, respectively. Now observe that repaired spare parts arrive at the stock point of the capital good according to a Poisson process with state dependent rates $\mu_r \cdot \min(S-n_0,K)$ whenever $n_r>0$. If there are no failed components in repair, then all spare parts are in stock and one is in operation which implies that $n_o = S$ and thus the arrivals from the repair shop to the stock point of the capital good are interrupted. This is indeed a closed queuing network with finite population $S$. This network admits a stationary probability distribution that can be written as a product form, only in the unrealistic setting that the components' lifetimes follow an exponential distribution.

Fortunately, the queuing network displayed in Figure \ref{fig:model1} corresponds precisely to an $M_n/GI/1$ queue with fixed capacity $S$ and state-dependent Poisson arrival rates $\mu_r \cdot \min(S-n_0,K)$. In fact, \cite{Lavenberg1975} showed that a (state-dependent) $M/GI/1$ queue with finite capacity can be viewed as a closed network and the opposite is also true, where the first queue has a general service time and the the second queue has an exponential service time. In this paper, we use this analogy to compute performance measures of the system, thereby allowing for a general lifetime distribution.  

Let $N(t)$ denotes the number of orders present in the $M_n/GI/1$ queue with fixed capacity $S$ at time $t$. The steady-state probability of having $n$ orders in this queuing system is denoted by $P(n\vert S,\tau, K) = \lim_{t\to\infty} \P(N(t) =n\vert S, \tau,K)$. Costs are levied in the following way. We penalize downtime of the capital good at rate $C_d>0$ per time unit. Furthermore, if a component fails before reaching age-threshold $\tau$, then a cost $C_u>0$ is charged. Per unit of turn-around stock, we pay a cost of $C_a$ per time unit, which can be interpreted as an investment rate. Finally, each part on inventory near the capital good costs $C_h>0$ per time unit, and each unit of capacity $K$ costs $C_w>0$ per time unit. 
We aim to minimize the total sum of these three cost rates. Hence, the optimization problem that we seek to solve is the following multi-variable,
non-linear, non-convex, mixed integer programming problem (\textsc{mipp}):
\begin{align}
(\textsc{mipp}) \quad & \min_{\{S,\tau,K\}} && C_u(1 - P(0\vert S,\tau, K) \mu \P(X<\tau) + C_d P(0 \vert S, \tau, K) + C_a S + C_w K \label{objective} \\
 & \mbox{subject to} &&  S, K \in \N, \quad \tau\geq 0, \nonumber
\end{align}
where 
\begin{align} \label{meanInstall}
\frac{1}{\mu}=\int_{0}^{\tau}(1-F(x))dx
\end{align} 
is the mean installation time of a component given that the age-threshold is equal to $\tau$  (we drop the dependency on $\tau$ for notational clarity). Observe that the value of the objective function \eqref{objective} of problem \textsc{mipp} is solely determined by the steady-state probability distribution $P(n\vert S,\tau, K)$, which we will analyze in the next section.

\subsection{Derivation of the steady-state probabilities}\label{sec4}
In this section, we focus on how to compute the steady-state probability  $P(n\vert S,\tau, K)$, which in turn relies on the random installation time of a component. Due to the nature of the age-based maintenance policy, we must first express this general distribution as the minimum between the random lifetime variable $X$ and the age-threshold $\tau$. The following Lemma provides this distribution and its Laplace transform. Observe that in this Lemma, we first treat $\tau$ as a random variable so that we can obtain the installation time distribution by using a degenerate distribution with all its mass at the value $\tau$.
\begin{lemma}\label{laplaceInstallation}
Let $f(x)$ and $F(x)$ be the PDF and the CDF, respectively, of the component's lifetime. Let $h(y)$ and $H(y)$ be the PDF and the CDF, respectively, of the random preventive age-threshold, which we denote by the random variable $Y$. Let $Z\sim min(X,Y)$, then the PDF and the CDF of $Z$, denoted by $g(z)$ and $G(z)$ respectively, are given by
\begin{align*}
G(z)&=F(z)+H(z)-F(z)H(z), \mbox{ and,} \\
g(z)&=G^{'}(z) =f(z) +h(z) -f(z) H(z) -F(z) h(z).
\end{align*}
If we let $G^{\ast}$ be the Laplace transform of $g$, then for the particular case that the preventive maintenance time is a degenerate distribution with all its mass at the age-threshold $\tau$, then:
\begin{align*}
G^{\ast}(s) = \int _{0}^{\tau}f(x)e^{-sx}dx.
\end{align*}
\end{lemma}
\begin{proof}
The proof follows from straightforward probability arguments. Observe that
\begin{align*}
1-G(z)   =\P(\min(X,Y)>Z)=\P(X>z)\P(Y>z)&=(1-H(z))(1-F(z))\\
        &=1-F(z)-H(z)+F(z)H(z).
\end{align*}
And for the Laplace transform of $g$, we have $G^{\ast}(s) =F^{\ast}(s) - \int _{\tau}^{\infty}f(x)e^{-sx}dx=\int _{0}^{\tau}f(x)e^{-sx}dx$.
\end{proof}
We now continue with deriving the steady-state probability $P(n\vert S,\tau, K)$ -- the key ingredients for evaluating the objective function of problem \textsc{mipp} -- and divide our analysis in two cases, depending on the value of the repair capacity $K$ compared to the turn-around stock $S$:
\begin{enumerate}
	\item $S\leq K$. In this case, there is ample capacity in the repair shop. Hence, as long as there are $n$ items on stock, the arrival rate at the first queue is exponential with mean $\lambda _{n}=\mu_r (S-n)$.
	\item $S> K$. In this case, the repair shop is truly capacitated: As long as there are $n$ items on stock, the arrival rate at the first queue is exponential with mean  $\lambda _{n}=\mu_r \min(S-n, K)$.
\end{enumerate}
The analysis in both cases is based on the approach of \cite{Kerner2008} and \cite{Economou2015}  and makes use of Lemma \ref{laplaceInstallation}.

\subsubsection{Steady-state probabilities under ample capacity}\label{subsec4.1}
This case corresponds to the situation where the repair shop has ample capacity, which means that the capacity $K$ does not influence the repair time of the defective parts. Let us first introduce a new random variable $R_n$ that represents the distribution of the remaining service time at an arrival instant of a repaired part who finds $n$ parts in the stock. Let $F_{n}^{\ast}(s)$ denote the Laplace transform of $R_n$. Then, by Theorem 3.1 of \cite{Economou2015}, we can compute $F_{n}^{\ast}(s)$ using the following recursive scheme:
\begin{eqnarray}\label{eq1}
F_{n}^{\ast}(s)=\left\{
  \begin{array}{ll}
G^{\ast}(s)&\mbox{if}\;n=0,\\
\frac{\lambda_{n}}{s-\lambda_{n}}\left[G^{\ast}(\lambda_{n})\frac{1-F_{n-1}^{\ast}(s)}{1-F_{n-1}^{\ast}(\lambda_{n})}-G^{\ast}(s))\right]&\mbox{if}\;n \geq 1.\\
\end{array}
\right.
\end{eqnarray}
This recursion initializes with the Laplace transform $G^{\ast}(s)$, which we introduced in Lemma \ref{laplaceInstallation}. Using the scheme in \eqref{eq1} and Corollary 3.1 by  \cite{Economou2015}, the expected remaining service time, denoted by  $R(n)=\mathbb{E}(R_n)$, at an arrival instant of a repaired part who finds $n$ parts in the system, can also be computed recursively as:

\begin{eqnarray}\label{eq2}
R(n)=\left\{
\begin{array}{ll}
\frac{\frac{1}{\mu}}{1-G^{\ast}(\lambda_{n})}-\frac{1}{\lambda_{n}}&\mbox{if}\;n=1,\\
\frac{G^{\ast}(\lambda_{n})R(n-1)}{1-F_{n-1}^{\ast}(\lambda_{n})}-\frac{1}{\lambda_{n}}+\frac{1}{\mu}&\mbox{if}\;n\geq2,\\
\end{array}
\right.
\end{eqnarray}
where  $\lambda _{n}=\mu_r (S-n)$ and  $1/\mu$ is the mean installation time of a component given that the age-threshold is equal to $\tau$, see Equation \eqref{meanInstall}. 

Using Corollary 4 of \cite{Abouee-Mehrizi2016}, we can now express the steady-state probabilities as:
\begin{align}\label{eq4}
&P(n\vert S,\tau, K)=\\
& \left\{
\begin{array}{ll}
\left[1+\left(1+\lambda_{S-1}R(S-1)\right)\frac{\lambda_{0}}{\lambda_{S-1}}\prod_{i=0}^{S-2}\frac{1-F_{i}^{\ast}(\lambda_{i+1})}{G^{\ast}(\lambda_{i+1})}\right.
\left.+\sum_{j=1}^{S-2}\frac{\lambda_{0}}{\lambda_{j}}\prod_{i=0}^{j-1}\frac{1-F_{i}^{\ast}(\lambda_{i+1})}{G^{\ast}(\lambda_{i+1})}\right]^{-1}&\mbox{if}\;n=0,\\
\frac{\lambda_{0} P(0\vert S,\tau, K)}{\lambda_{n}}\prod_{i=0}^{n-1}\frac{1-F_{i}^{\ast}(\lambda_{i+1})}{G^{\ast}(\lambda_{i+1})}&\mbox{if}\leq n\leq S-1,\\
1-\sum_{j=0}^{S-1} P(j\vert S,\tau, K)&\mbox{if}\;n=S.\\
\end{array} \nonumber
\right.
\end{align}
Using \eqref{eq4}, we are now able to efficiently calculate the steady-state probabilities for the case of ample capacity at the repair shop.

\subsubsection{Steady-state probabilities for capacitated systems}
The derivation of the steady-state probabilities for the case that the repair shop is capacitated, i.e. when $S>K$, is slightly different from that where $K\leq S$. In fact, when all repairmen are busy, i.e. when $K$ parts being repaired at the repair shop, the arrival rate at the stock, $\lambda _{n}=\mu_r min(S-n, K)$, becomes constant and is equal to $K\mu_r$. The probabilities we derived in the previous section cannot be applied directly because they involve the Laplace transform of $R_n$ which turns to zero if $\lambda_n$ becomes constant. Fortunately, Theorem 6 by \cite{Oz2017} allows us to find $F_{n}^{\ast}(s)$ even for constant arrival rate. We restate their theorem using our notations. Let $f_{n}(x)=\P(R_n=x)\ (x\geq0)$ denote the probability density function of the random variable $R_n$, \cite{Oz2017} showed that the conditional residual service time  $f_{n}(x)$ satisfies the following recursion:
\begin{eqnarray}\label{eq5}
f_{n}(x)=\left\{
\begin{array}{ll}
g(x)&\mbox{if}\;n=0,\\
 \lambda e^{\lambda x} \int _{x}^{\infty}e^{-\lambda x_{0}}g(x_{0})dx_{0}+\frac{G^{\ast}(\lambda)\left(\lambda e^{\lambda x}\int _{x}^{\infty}e^{-\lambda x_{0}}f_{n-1}(x_{0})dx_{0}\right)}{1-F_{n}^{\ast}(\lambda)}&\mbox{if}\;n\geq 1.\\
\end{array}
\right.
\end{eqnarray}
Using the above recursion at $n=0,...,S-K$ and  $\lambda=\lambda _{n}=K\mu_r$, we can rewrite $f_{n}(x)$ and $F_{n}^{\ast}(s)$ for $n=0,1,$ and $2$ as follows:
\begin{eqnarray}\label{eq6}
\left\{
\begin{array}{l}
f_{0}(x)=g(x),\\
F_{0}^{\ast}(s)=\int_{0}^{\infty}f_{0}(x)e^{-sx}dx=G^{\ast}(s),\\
	f_{1}(x)=\frac{\lambda e^{\lambda x}\int_{x}^{\infty}(e^{-\lambda x_{0}}g(x_{0})dx_{0}}{1-F_{0}^{\ast}(\lambda)},\\
	F_{1}^{\ast}(s)=\frac{\lambda\int_{0}^{\infty}(\int_{x}^{\infty}e^{-sx_{0}}g(x_{0})dx_{0})dx}{1-F_{0}^{\ast}(\lambda)}=\frac{\lambda\int _{0}^{\infty}xe^{-sx}g(x)dx}{1-F_{0}^{\ast}(\lambda)}=\frac{\lambda }{1-F_{0}^{\ast}(\lambda)}\frac{d~F_{0}^{\ast}(s)}{ds},\\
		f_{2}(x)=\lambda e^{\lambda x}\int_{x}^{\infty}e^{-\lambda x_{0}}g(x_{0})dx_{0}+G^{\ast}(\lambda)\frac{\lambda e^{\lambda x}\int_{x}^{\infty}e^{-\lambda x_{1}}f_{1}(x_{1})dx1}{1-F_{1}^{\ast}(\lambda)},\\
	F_{2}^{\ast}(s)=-\lambda\frac{dG^{\ast}(s)}{ds}-\frac{\lambda^{2}G^{\ast}(\lambda)}{(1-F_{1}^{\ast}(\lambda))F_{0}^{\ast}(\lambda) }\frac{d^{2}F_{0}^{\ast}(s)}{2ds^{2}},\\
\end{array} \nonumber
\right.
\end{eqnarray}
and for the general case,  we can write for $n\geq1$:
\begin{eqnarray}
F_{n}^{\ast}(s)=\frac{\lambda}{(1-F_{n-1}^{\ast}(\lambda))}\left[-G^{\ast}(\lambda)\frac{dF_{n-1}^{\ast}(s)}{ds}+\frac{dG^{\ast}(s)}{ds} (F_{n-1}^{\ast}(\lambda)-1)\right],\nonumber
\end{eqnarray}
where,
\begin{eqnarray}
\left\{
\begin{array}{l}
\frac{dF_{n-1}^{\ast}(s)}{ds}=\frac{\lambda}{(1-F_{n-2}^{\ast}(\lambda))}\frac{1}{2}\left[-G^{\ast}(\lambda)\frac{d^{2}F_{n-2}^{\ast} (s)}{ds^{2}}+\frac{d^{2}G^{\ast}(s)}{ds^{2}}(F_{n-2}^{\ast}(\lambda)-1)\right],\\
	\frac{d^{2}F_{n-2}^{\ast}(s)}{ds^{2}}=\frac{\lambda}{(1-F_{n-3}^{\ast}(\lambda))}\frac{1}{3}\left[-G^{\ast}(\lambda)\frac{d^{3}F_{n-3}^{\ast} (s)}{ds^{3}}+\frac{d^{3}G^{\ast}(s)}{ds^{3}}(F_{n-3}^{\ast}(\lambda)-1)\right],\\
	\vdots\\
	\frac{d^{n-1}F_{1}^{\ast}(s)}{ds^{n-1}}=\frac{\lambda}{(1-F_{0}^{\ast}(\lambda))}\frac{1}{n}\left[-G^{\ast}(\lambda)\frac{d^{n}F_{0}^{\ast}( \lambda)}{ds^{n}}+\frac{d^{n}G^{\ast}(s)}{ds^{n}}(F_{0}^{\ast}(\lambda)-1)\right].\\
\end{array}
\right.\nonumber
\end{eqnarray}
In summary, for a repair shop  with limited capacity $K<S$, the Laplace transform of the random variable that represents the distribution of the remaining service time at an arrival instant of a customer who finds $n$ customers in the system, $R_n$ can be written as:
\begin{eqnarray}\label{eq10}
F_{n}^{\ast}(s)=\left\{
\begin{array}{ll}
G^{\ast}(s)&\mbox{if}\;n=0,\\
\frac{\lambda}{(1-F_{n-1}^{\ast}(\lambda))}(-G^{\ast}(\lambda)\frac{dF_{n-1}^{\ast}(s)}{ds}+\frac{dG^{\ast}(s)}{ds}(F_{n-1}^{\ast}(\lambda) -1)&\mbox{if}\;s=\lambda_{n},\\
\frac{\lambda_{n}}{s-\lambda_{n}}(G(\lambda_{n})\frac{1-F_{n-1}(s)}{1-F_{n-1}(\lambda_{n})}-G(s))&\mbox{if}\;s\neq\lambda_{n},\\
\end{array}
\right.
\end{eqnarray}
and the steady-state probability is again given by Equation \eqref{eq4}, where $F_{n}^{\ast}(s)$ is now given by \eqref{eq10}.

Now that we have derived expressions for computing the steady-state probabilities, we turn our attention back again to the objective function of problem \textsc{mipp}. This is the aim of next section.

\subsection{Cost function properties and optimizing algorithms}
In this section we investigate the objective function of problem \textsc{mipp} in more detail. Recall that this objective function is given by 
$$\quad  \min_{\{S,\tau,K\}}   C_u(1 - P(0\vert S,\tau, K) \mu \P(X<\tau) + C_d P(0 \vert S, \tau, K) + C_a S + C_w K,$$
with  $\{S,K\}\in\N$ and $\tau\in\R^+$. Unfortunately, we are unable to show that this function is convex on one of these parameters, namely $S$, $\tau$ or $K$, which would have eased the procedure of finding the optimal values significantly. Fortunately, however, we are able to demonstrate other desireable properties concerning the monotonicity of the probability of being in the state zero, i.e. $n=0$. Specifically, as the following Lemma illustrates, we have that $P(0\vert S,\tau, K)$ is decreasing (in the weak sense) in $S$ and $K$.
\begin{lemma}\label{monotoneProb}
The following two statements hold:\\
(i) $P(0\vert S,\tau, K)$ is non-increasing in $S$.\\
(ii) $P(0\vert S,\tau, K)$ is non-increasing in $K$.
\end{lemma}
\begin{proof}
We first rewrite $P(0\vert S,\tau, K)$.
From Equation 3.23 of \cite{Economou2015} and Equation 1 in \cite{Ross2006}, we know that for all $n\geq2$
\begin{eqnarray}\nonumber
\sum_{j=n}^{S} P(j\vert S,\tau, K)= P(n-1\vert S,\tau, K)\lambda_{n-1}R(n-1)+\mu\sum_{j=n}^{S}\lambda_{j} P(j\vert S,\tau, K).
\end{eqnarray}
For $n=2$, we know that $\sum_{j=1}^{S} P(j\vert S,\tau, K)=1- P(0\vert S,\tau, K)- P(1\vert S,\tau, K)$, and also that $\lambda_{S}=0$ for $n=S$. Hence,
\begin{align}\nonumber
\sum_{j=1}^{S} P(j\vert S,\tau, K)=1- P(0\vert S,\tau, K)- P(1\vert S,\tau, K)= P(1\vert S,\tau, K)\lambda_{1}R(1)+\frac{1}{\mu}\sum_{j=1}^{S-1}\lambda_{j} P(j\vert S,\tau, K).
\end{align}
Using the expression of $R(1)$ and $ P(1\vert S,\tau, K)$ from Equation \eqref{eq2} and \eqref{eq4} yields
\begin{eqnarray}\nonumber
 P(0\vert S,\tau, K)=
\left[1+\frac{\lambda_{0}}{\mu G^{\ast}(\lambda_{1})}+\frac{\lambda_{0}}{\mu}\sum_{n=2}^{S-1}\prod_{i=0}^{n-1}\frac{1-F_{i}^{\ast}(\lambda_{i+1})}{G^{\ast}(\lambda_{i+1})}\right]^{-1}.
\end{eqnarray}
Consider now two node cyclic networks as represented in Figure \ref{fig:model1}: One with $S$ and one with $S+1$ customers. All other parameters are the same for both networks. By Theorem 1 (iii) of \cite{Shanthikumar1989}  we know that the average output rate at node $i, i=1,2$ of network one with a population of $S$ customers is less than the average output rate at of the same node of network two with a population of $S+1$. We can rewrite this result as
\begin{eqnarray}\nonumber
\sum_{j=0}^{S-1}\lambda_{j}P(j\vert S,\tau, K)=\mu(1-P(0\vert S,\tau, K)\leq \sum_{j=0}^{S}\lambda_{j}P(j\vert S+1,\tau, K)=\mu(1-P(0\vert S+1,\tau, K)),
\end{eqnarray}
which proves Assertion (i). Similarly, to prove Assertion (ii), we consider two  node cyclic networks: One with server capacity $K$ and one with server capacity $K+1$, respectively, at node 2. Then, by Theorem 3 of \cite{Shanthikumar1989}, we have:
\begin{eqnarray}\nonumber
\mu(1-P(0\vert S,\tau, K)\leq \mu(1-P(0\vert S,\tau, K+1)),
\end{eqnarray}
which completes the proof of Assertion (ii).
\end{proof}
Lemma \ref{monotoneProb} is not only intuitive -- indeed, increasing repair capacity and/or the turn-around stock will not lower the probability of downtime -- it also helps us in designing algorithms to compute optimal values of $K$ and $S$ for a fixed choice of $\tau$. Specifically, from Equation (1) and Lemma \ref{monotoneProb}, it is clear that $C_u(1 - P(0\vert S,\tau, K) + C_a S$ is increasing in $S$. Moreover, this expression tends to infinity as $S$ tends to infinity. Based on these observations, we propose our first algorithm, Algorithm \ref{Algo0}, which enables to find the optimal $S^*$ (for a fixed value of $K$ and $\tau$).

Based on Algorithm \ref{Algo0}, we propose Algorithm \ref{Algo1} that makes use of Algorithm \ref{Algo0}, that is guaranteed to find the optimal value of $K^*$ and $S^*$ for a fixed value of $\tau$. Indeed, according to Lemma 1, the steady-state probability $P(0\vert S,\tau, K)$  is decreasing as $K$ increases. Therefore, we can use the Algorithm \ref{Algo0} recursively to find the optimal value of $K^*$ for a given $S$ (and hence also for $S^*$). 

Observe that in both algorithms, we repeatedly need to compute the steady-state probabilities for fixed choices of $S$, $K$, and $\tau$. For these computations, we make use of the closed-form formulas developed in Section \ref{sec4}.

\begin{algorithm}[H]\label{Algo0}
\setstretch{1.1}
\footnotesize
\SetAlgoLined
\KwResult{Optimal $S^*$}
 Step 1: Fix $S:=0, S^{*}:=0 $ and $TC_{min}=\infty$\\
 Step 2: Compute $TC=C_u(1 - P(0\vert S,\tau, K) \mu \P(X<\tau) + C_d P(0 \vert S, \tau, K) + C_a S + C_w K.$\\
 \If{$TC<TC_{min}$}{
  $TC_{min}=TC$\;
   $S:=S+1$\;
   Go to step 3\;
   }
   Step 3: Compute $C_u(1 - P(0\vert S,\tau, K) \mu \P(X<\tau)  + C_a S + C_w K$\\
    \eIf{$C_u(1 - P(0\vert S,\tau, K) \mu \P(X<\tau)  + C_a S + C_w K<TC_{min}$}{
   Go to Step 2\;
   }{Stop.}
   \caption{Finding the optimal stock level $S^*$ for a given number of repairmen $K$.}
\end{algorithm}

\begin{algorithm}[H]\label{Algo1}
\setstretch{1.1}
\footnotesize
\SetAlgoLined
\KwResult{Optimal $S^*$ and $K^*$}
 Step 1: Fix $K:=1$ use Algorithm \ref{Algo0}  to find $S^*$ and its corresponding  optimal total cost $TC(S^*,1,\tau)$. Set $TC_{min}=TC$\\
 Step 2: Set $K:=K+1$, use Algorithm \ref{Algo0} and find the new $S^*$ and it corresponding optimal total cost $TC(S^*,K,\tau)$.\\
 \If{$TC<TC_{min}$}{
  $TC_{min}=TC$\;
   Go to step 3\;
   }
   Step 3: Compute $C_u(1 - P(0\vert S^*,\tau, K) \mu \P(X<\tau)  + C_a S^* + C_w K$\\
    \eIf{$C_u(1 - P(0\vert S^*,\tau, K) \mu \P(X<\tau)  + C_a S^* + C_w K<TC_{min}$}{
   Go to Step 2\;
   }{Stop.}
   \caption{Finding the optimal stock level $S^*$ and the optimal capacity $K^*$ .}
\end{algorithm}

\section{Extension to multiple capital goods} \label{multipleCapital}
In this section, we use our exact analysis for the single capital good, as described in the previous sections, to design an approximation for the case where we have several capital goods.

We now consider a network in which each capital good is characterized by different failure rates and age-thresholds, but where they share a common repair shop with limited repairman resources. This network is shown in Figure \ref{fig:model2}. Because the repair shop resources are shared among the capital goods, we need to decide upon the repair prioritization policy in the repair shop. We consider both a first-come, first-served (FCFS) discipline, where parts are repaired in the order in which they arrive at the repair shop, and a priority scheduling policy with priority classes. When considering a priority scheduling policy, the case where the number of repairman is greater than one is unfortunately not tractable to analyse. In that setting, the steady-states probabilities for each class of capital goods cannot be computed due to the curse of dimensionality caused by multiple classes and multiple servers. Therefore, we only treat the case with priorities under the assumption of a single repair resource at the repair shop, while for the FCFS discipline, we also treat the case with multiple repair resources.  

In summary, we study two systems:
\begin{enumerate}
\item A system for which the repair shop is an $\bullet/M/K$ queue with FCFS discipline.
\item A system for which the repair shop is an $\bullet/M/1$ with priorities between classes.
\end{enumerate} 

 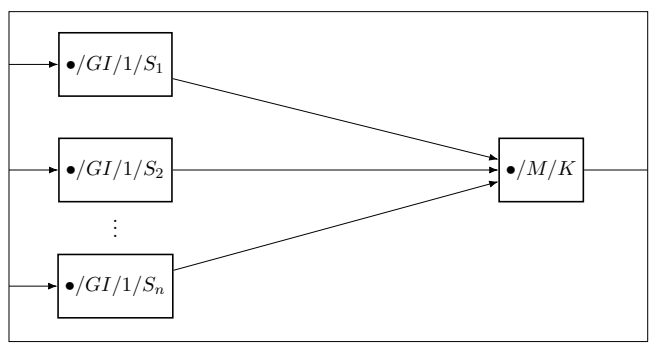
\begin{figure}[ht!]
% \centering\includegraphics[width=13cm, height=7cm]{Model2.eps}
\centering
\scalebox{0.7}{%
\begin{tikzpicture}
     \node[rectangle,draw](a) at (0,4)[draw,line width = 0.3mm,minimum width=1.2cm,minimum height=1.2cm] {$\bullet/GI/1/S_1$};
 \node[rectangle,draw](b) at (0,2) [draw,line width = 0.3mm,minimum width=1.2cm,minimum height=1.2cm] {$\bullet/GI/1/S_2$};
  \node (e) at (0,1) {$\vdots$};
 \node[rectangle,draw] (c) at (0,-0.2)[draw,line width = 0.3mm,minimum width=1.2cm,minimum height=1.2cm] {$\bullet/GI/1/S_n$};
  \node[rectangle,draw] (d) at (8,2)[draw,line width = 0.3mm,minimum width=1.2cm,minimum height=1.2cm] {$\bullet/M/K$};
  
 \draw[-Latex] (a) -- (d);
 \draw[-Latex] (b) -- (d);
\draw[-Latex] (c) -- (d);
\draw[-] (10,-1.25)-- (-2,-1.25)-- (-2,5)-- (10,5)-- (10,-1.25);
 %\draw[->] (10,5)-- (-2,5)-- (-2,5)--(-1,5);
 %\draw[-] (-2,0) -- (-2,5);
  \draw[-] (d) -- (10,2);
\draw[-Latex] (-2,-0.2) -- (c);
\draw[-Latex] (-2,4) -- (a);
\draw[-Latex] (-2,2) -- (b);
 \end{tikzpicture}}
     \caption{Cyclic maintenance network with multiple capital goods}
    \label{fig:model2}
\end{figure}

The networks we consider do not have a steady-state probability that can be written as product form of state probabilities of each node, so we use the approximation method of \cite{Marie1979} to obtain these probabilities. The method developed by \cite{Marie1979} and extended by \cite{Bayna1996} consists of replacing general service time queues with state-dependent exponential service time queues.  The exponential service time is obtained by analysing each station with a general service time in isolation but with a state-dependent arrival rate.  The method is iterative and appears to be a good technique for providing different performance measures (indeed, we later compare this method with a simulation which shows that the method performs very well). \cite{Bayna1996} extended Marie's work to multiple classes by simply decomposing the original network into multiple networks, each of which handles only one class. The dependency between classes is taken into account when analysing isolated multi-class stations. 

\subsection{Repairing under a first-come, first-served discipline}
We first replace for each capital good $i$ the $\bullet/G_i/1/S_i$ by an $\bullet/M(n)/1/S_i$ queue with state dependent service rate $\mu_{i}(n), n=1,...,S_i$, and the repair shop $\bullet/M/K$ queue by $\bullet/M_n/K$ node node with state dependent service rate $\mu_{i,0}(n), n=1,...,S_i$. Our network associated with capital good $i$ is now a two nodes Jackson network as shown in Figure \ref{fig:subnetwork}. This network does admit a product stationary probability, which can be written as follows:
\begin{align}\label{prodnetwork0}
\tilde{P}_{i}(n)=T_i \left (\prod _{k=1}^n \frac{1}{\mu_{i}(k)} \right )\prod _{k=1}^{S_i-n} \frac{1}{\mu_{i,0}(k)}, i=1,...,J;n=0,...,S_i
\end{align}
where, $T_i$ is the normalising constant of the node of capital good $i$, $n$ are the number of items in the queue $\bullet/G_i/1/S_i$ and $S_i-n$ are the number of failed parts at the repair shop node. Observe that this sub-network in isolation is similar to the model we analysed in Section \ref{sec4}. 

Since capital good $i,i=1,...,J$, can only have capital good $i$ items, the marginal steady-state probability to have $n$ items at capital good $i$ is $\tilde{P}_{i}(n)$. For the repair shop node, the marginal steady-state probability to have $n$ items of class $i$ is then simply $\tilde{P}_{i}(S_i-n)$.

\begin{figure}[ht!] \centering
\fontsize{8pt}{10pt}\selectfont
\begin{tikzpicture}[
    triangle/.style = {draw, line width = 0.3mm, regular polygon, regular polygon sides=3, minimum size=2cm},
    node rotated/.style = {rotate=180},
    border rotated/.style = {shape border rotate=180},
  ]
	\node (rect2) at (2,1.5) [draw,line width = 0.3mm,minimum width=1.3cm,minimum height=1.1cm]  {$\bullet/M_n/1/S_i$};
	\node (rect1) at (5,1.5) [draw,line width = 0.3mm,minimum width=1.3cm,minimum height=1.1cm] {$\bullet/M_n/K$};
	\node [text width=3cm, align=center, inner sep=0pt] at (2,0.5) {Capital good};
	\node [text width=3cm, align=center, inner sep=0pt] at (5,0.5) {Repair shop};
	\draw [-Latex, line width = 0.3mm] (0,1.5) -- (rect2.west);
	\draw [-Latex, line width = 0.3mm] (rect2.east) -- (rect1.west);
	\filldraw [ line width = 0.3mm] (rect1.east) -- (7,1.5);
	\draw [ line width = 0.3mm] (7,1.5) to (7,0);
	\draw [ line width = 0.3mm] (7,0) -- (0,0);
	\draw [ line width = 0.3mm] (0,0) -- (0,1.5);
\end{tikzpicture}
     \caption{Sub-network for capital good $i$.}
    \label{fig:subnetwork}
\end{figure}
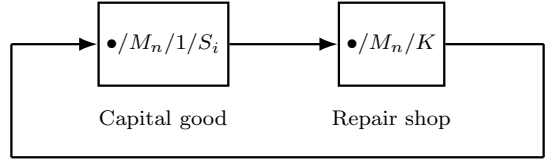

For this Jackson network, we can readily obtain the state dependent arrival rate $\lambda_{i}(n)$ at each node $i, i=1,...,J$ and the the state dependent arrival rate at the repair shop (node 0) $\lambda_{0,i}(n)$ for each class $i, i=1,...,J,$ using the following equations:

\begin{align}\label{arrivalrate}
\lambda_{i}(n) &=\mu_{i}(n+1)\frac{\tilde{P}_{i}(n+1)}{\tilde{P}_{i}(n)},\text{ and, }\lambda_{0,i}(n) &=\mu_{0,i}(n+1)\frac{\tilde{P}_{i}(S_i-n-1)}{\tilde{P}_{i}(S_i-n)}, i=1,...,J.
\end{align}
However, the scheme above in \eqref{arrivalrate} can only be used if the service rates $\mu_{i}(n)$ and $\mu_{0,i}(n)$, $n=1,...,S_i$ are known.  It turns out that these service rate can be obtained by analysing each node in isolation for a given values of $\lambda_{i}(n)$ and $\lambda_{0,i}(n)$, for which we have derived exact expressions in Section \ref{sec4}.

Specifically, for nodes $i,i=1,...,J$, the steady-state probabilities, which we now index by $i$,  can be used as a way to approximate the state dependent service rate at node $i$. The principle of this approximation is to set the arrival rate equal to $\lambda_{i}(n), n=0,...,S_i$ for node $i,i=1,...,J$. Then by Equation \eqref{eq4}, we can approximate the conditional throughput (average output rate at node $i$), denoted by $\upsilon_i$  as
\begin{align*}
\upsilon_i(n)=\lambda_{i}(n-1)\frac{P_{i}(n-1)}{P_{i}(n)}, n=1,...,S_i,
\end{align*}
Now, the approximate state dependent service rate at node $i$ is set to
\begin{align}\label{service-rate-i}
\mu_i(n)=\upsilon_i(n), n=1,...,S_i
\end{align}

We now need to find the state dependent service rate $\mu_{0,i}(n)$ at the repair shop (node 0). To do so, we again assume that we know the state dependent arrival rates $\lambda_{0,i}(n)$. We can then obtain $\mu_{0,i}(n)$ in a similar way as for nodes $1,...,J$. 

Since there are $J$ classes of items to be repaired at node 0, the repair shop can be viewed as a multi-class queueing system. If it has $K$ repairman, then the steady-state probability of this queueing system is given by:
\begin{align}\label{eq25}
P_0(n_1,...,n_J)=\begin{cases}\frac{1}{T_0}(1/\mu_r)^{\sum_{i=1}^{J}n_i}\prod_{i=1}^{J}\frac{\prod_{k=0}^{n_i-1}\lambda_{0,i}(k)}{n_i!}, & \sum_{i=1}^{J}n_i \leq K,\\
\frac{1}{T_0}\frac{\left(\sum_{i=1}^{J}n_i\right)!(1/\mu_r)^{\sum_{i=1}^{J}n_i}}{K!K^{\left(\sum_{i=1}^{J}n_i\right)-K}}\prod_{i=1}^{J}\frac{\prod_{k=0}^{n_i-1}\lambda_{0,i}(k)}{n_i!}, &\sum_{i=1}^{J}n_i > K,\end{cases}
\end{align}
where $T_0$ is the normalising constant and $n_i, i=1,...,J,$ are the number of failed parts of capital good $i$ at the repair shop.

The marginal steady-state probability of having $n_i, i=1,...,J,$ parts at the repair shop can be expressed as
\begin{align}\label{eq26}
P_0(n_i)=\sum_{n_j\leq S_j, j\neq i}P_0(n_1,...,n_J), n_i=0,...,S_i.
\end{align}
From the above equation, the conditional throughput of class $i$ at the repair shop node is 
\begin{align*}
\upsilon_{0,i}(n)=\lambda_{0,i}(n_i-1)\frac{P_{0}(n_i-1)}{P_{0}(n_i)}, n_i=1,...,S_i,
\end{align*}
and the approximate state dependent service rate of class $i$ at node $0$ is set as
\begin{align}\label{eq27}
\mu_{0,i}(n)=\upsilon_{0,i}(n), n=1,...,S_i.
\end{align}

We are now ready to provide an algorithm, Algorithm \ref{algo3}, to compute the steady-state probabilities of the original network of Figure \ref{fig:model2} and provide an approximate solution for finding the best $S_i^*, i=1,...J.$ that minimizes the total cost. This algorithm makes use of Algorithm \ref{Algo1}, we developed in Section \ref{sec4}.

\begin{algorithm}[H]\label{algo3}
\setstretch{1.1}
\SetAlgoLined
\footnotesize
\KwResult{Optimal $S_i^*$}
 Step 1:
 \For{i=1:J}{
  Set $\mu_{0,i}(n_i)=n_i\mu_r$ and $\mu_{i}(n)=\mu_{i}$\;This is an initialisation for service rate at each node, $\mu_{i}$ is given by Equation \eqref{meanInstall}\;
  Set $v=1$ and $\mu_{0,i}^v(n_i)=\mu_{0,i}(n_i)$ , $\mu_{i}^v(n)=\mu_{i}(n)$ and fix some tolerance $\epsilon$;}
   Step 2:
 \For{i=1:J}{
  Compute $\tilde{P}_{i}(n), n=0,...,S_i$ by Equation \eqref{prodnetwork0} and using $\mu_{0,i}^r(n_i)$ and $\mu_{i}^r(n)$\;
  Compute $\lambda_{0,i}(n_i), n_i=0,...,S_i-1$ and $\lambda_{i}(n), n=0,...,S_i-1$ using Equation \eqref{arrivalrate};}
 Step 3:
 \For{i=1:J}{
  Compute $P_{i}(n), n=0,...,S_i$ using Equation (\ref{eq4})\;
  Compute $P_{0}(n_i), n_i=0,...,S_i$ using Equations (\ref{eq25}) and (\ref{eq26}) \;
  Compute $\mu_{0,i}(n_i), n_i=0,...,S_i-1$ and $\mu_{i}(n), n=0,...,S_i-1$ using Equations \eqref{service-rate-i} and \eqref{eq27}\;
  Set $v=v+1$ $\mu_{0,i}^v(n_i)=\mu_{0,i}(n_i)$ and $\mu_{i}^v(n)=\mu_{i}(n)$}
 \eIf{$\mu_{0,i}^v(n_i)-\mu_{0,i}^{v-1}(n_i)> \epsilon$}{
  Go to step 2\;}{
  Use Algorithm \ref{Algo1} to find the optimal $S_i^*, i=1,...,J$.}
   \caption{Finding the optimal $S_i^*, i=1,...,J,$ for a given number of repairman $K$.}
\end{algorithm}

%********* Here I cannot go further, since queue with priority and state dependent arrival rates are difficult to analyse. I did not find any work that deals with state dependent arrival M/M/1 queue and multiple priorities.**********
%In case of non state dependent arrival rates, we use the work of \cite{Sleptchenko2015}.******

\subsection{Repairing under a priority discipline}
We now treat the case where the repair station is an $\bullet/M/1$ queue with multiple priorities. We consider the case where the station station a single server (repairman) and that the failed parts are ranked in their order of priority. We assume that class 1 has the highest priority and class $J$ has the lowest priority. When a failed component of type $i<j$ arrives at the repair station and the part $j$ is being repaired, the server immediately stops repairing the component $j$ and starts with part $i$. The part $j$ resumes the repair at the point where it stopped if no part of a higher class is being repaired or in the queue. That is, we assume a preemptive priority rule.

A queue with multiple priorities and a preemptive service discipline has been analysed very well in the literature. The majority of works approximate the stationary probability or provides some performance measures without computing the steady-state probabilities (e.g. mean response time). We refer the reader to the work of \cite{wu2014classification}, where the most important works have been largely summarized under the assumption that the arrival rate is independent of the state of the system. 

For state-dependent arrival rates, little work has been done. The most recent is that of \cite{Brandwajn2017}, who developed an algorithmic procedure to approximate the stationary probability of each priority class. Only the work by \cite{Bitran2002} gives the exact calculation of the stationary probabilities with two priority classes.  We summarize their procedure in the following Lemma. 
\begin{lemma}\label{lemma3}\citep{Bitran2002},
For a two class $M_n/M/1$ queue with state dependent arrival rates, the steady-state probabilities are:\\
\begin{align*}
P_1(n)=\begin{cases}\left[\sum_{n=1}^{S_1}\left((1/\mu_r)^k\prod_{k=0}^{n-1}\lambda_1(k) \right)\right]^{-1} & n = 0\\P_1(0)(1/\mu_r)^k\prod_{k=0}^{n-1}\lambda_1(k) & n=1,...,S_1\end{cases} 
\end{align*}
and
\begin{align*}
P_2(n)=\frac{e'C_k}{\sum_{n=0}^{S_2}e'C_k}
\end{align*}
where:\\
 $C_0$ is the right eigenvector of $B_0$ associated to the eigenvalue 0,\\
$ B_k=A_k-\lambda_2(k)e_1e',k=0,...,S_2$\\
$e'$ is the transpose of a column matrix with ones elements and size $S_1+1$,\\
$e_1$ is a column matrix of size $S_1+1$ with one in its first element and zeros otherwise,\\
$ B_k=A_k-\lambda_2(k)e_1e',k=0,...,S_2$\\
\begin{align*}
A_k=\begin{pmatrix}
a_{0,k} & -\mu_r &0& \cdots &0 \\
-\lambda_{1}(0) & a_{1,k} &-\mu_r&0 \cdots & 0 \\
  & \ddots  & \ddots & \ddots  \\
& & -\lambda_{1}(S_1-2) & a_{S_1-1,k} &-\mu_r\\
0&\cdots &\cdots& -\lambda_{1}(S_1-1) & a_{S_1,k} \\
\end{pmatrix}
\end{align*}

\begin{align*}
a_{i,j}=\begin{cases}\lambda_1(0)+\lambda_2(0) & i=j= 0\\
\lambda_1(0)+\lambda_2(j) +\mu_r& i,j> 0\\
\lambda_2(j) +\mu_r& i=S_1,j\geq  0\\
\lambda_1(i)+\lambda_2(j) +\mu_r& \\
\end{cases}
\end{align*}
and $C_k=\lambda_2(k-1)B_k^{-1}C_{k-1}, k=1,...,S_2$
\end{lemma}

We now provide an extension of Lemma \ref{lemma3} to more than two classes. Observe that class $j$ $(2<j\leq J)$, sees all priorities classes $i< j$ as a single priority class with a Poisson arrival rate
\begin{align*}
\lambda=\sum_{i=1}^{j-1}\sum_{i=0}^{S_i}\lambda_i(n)P_i(n).
\end{align*}
Therefore, to find the steady-state probability $P_j(n), n=0,...,S_j, 2<j\leq J$, we can apply Lemma \ref{lemma3} by setting $\lambda_1(n)=\lambda$ and $S_1=\infty$. We will use this observation in computing the steady-state probabilities for the system for which the repair shop is an $\bullet/M/1$ with priorities between classes.

%\begin{align}\label{eq29}
%P_{0}(n_i)=\begin{cases}(1-\rho_i^h-\rho_i)+\frac{\rho_i^h}{\rho_i}(1-\rho_i^h-\rho_i)(1-g_i(n_i)& n_i = 0\\
%\rho_iP_{0}(n_i-1)+\rho_i^h\sum_{l=0}^{n_i-1}\left(P_{0}(n_i-1-l)\left[1-\sum_{v=0}^{l}g_i(v)\right]\right)+\\\frac{\rho_i^h}{\rho_i}(1-\rho_i^h-\rho_i)\left[1-\sum_{v=0}^{n_i}g_i(v)\right] & n_i > 0\end{cases}
%\end{align}
%where,

%\begin{align}\label{eq30}
%g_{i}(n_i)=\begin{cases}\frac{(1+\rho_i^h+\rho_i)-\sqrt{(1+\rho_i^h+\rho_i)^2-4\rho_i^h}}{2\rho_i^h},& n_i = 0\\
%\frac{\rho_ig_{i}(n_i-1)+\rho_i^h\sum_{l=1}^{n_i-1}g_{i}(l)g_{i}(n_i-l)}{(1+\rho_i^h+\rho_i)-2\rho_i^hg_{i}(0)},& n_i > 0\\
%\end{cases}\\
%\end{align}

%and
%\begin{align}
%\rho_i^h=\sum_{l=1}^{i-1}\lambda_l
%\end{align}

\section{Numerical investigation}
\label{sec:numericalResults}
In this section we report on an extensive numerical investigation. We first discuss our experiments with one capital good, and then discuss those involving multiple capital goods. The objectives of the numerical experiments differ per setting.

\subsection{One capital good}
We start with evaluating the performance of preventive maintenance in the context of a single capital good through comparative statics. We first report on Figures \ref{fig4}-\ref{fig7} the total cost as a function of $\tau$ and the capacity of the repair shop $K$. We vary $\tau$ between 0 and 2 and we consider values of $K$ between 1 and 5. In order to analyze the sensitivity to the unit costs, we consider values of $C_u \in \{10, 20\}$, $C_d \in \{20, 40\}$ and $C_w = 0.75$.

\begin{figure} [h]
 \begin{minipage}[c]{0.48\textwidth}
 \label{fig:realnework}
  \centering\epsfig{figure=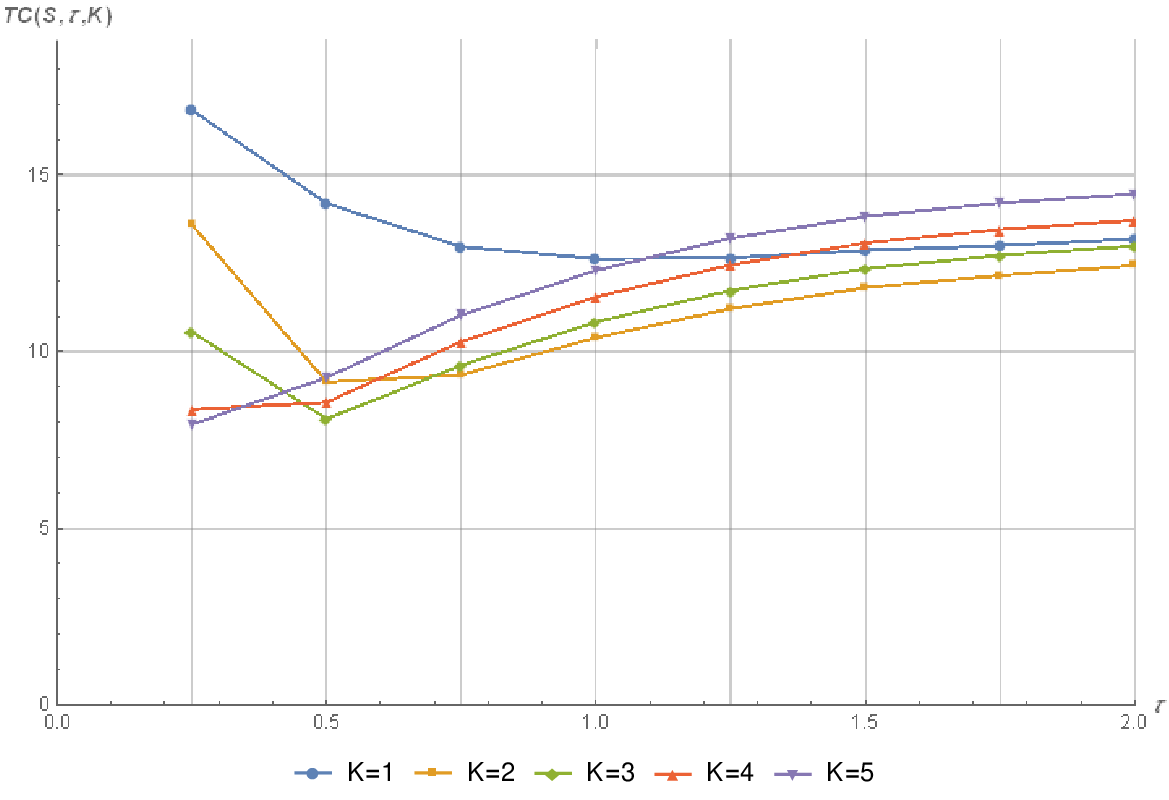, width=8.5cm, height=6cm}
  \caption{Total cost variation as a function of $\tau$ and the capacity $K$, $C_u=10$, $C_d=20$, $C_w=0.75$.}\label{fig4}
 \end{minipage} \hfill
 \begin{minipage}[c]{0.48\linewidth}
 \label{fig:SZnetwork}
  \centering\epsfig{figure=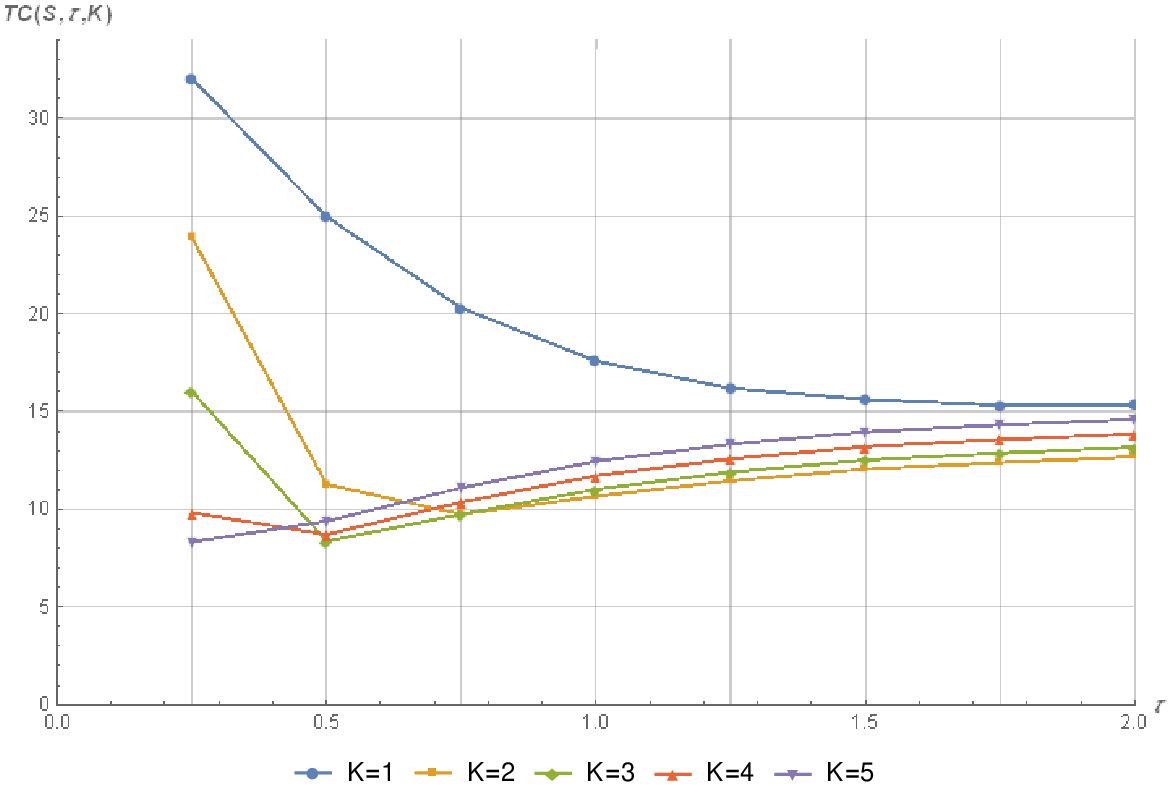, width=8.5cm, height=6cm}
  \caption{Total cost variation as a function of $\tau$ and the capacity $K$, $C_u=10$, $C_d=40$, $C_w=0.75$.}\label{fig5}
 \end{minipage}
\end{figure}

\begin{figure}[h]
 \begin{minipage}[c]{0.48\textwidth}
 \label{fig:realnework}
  \centering\epsfig{figure=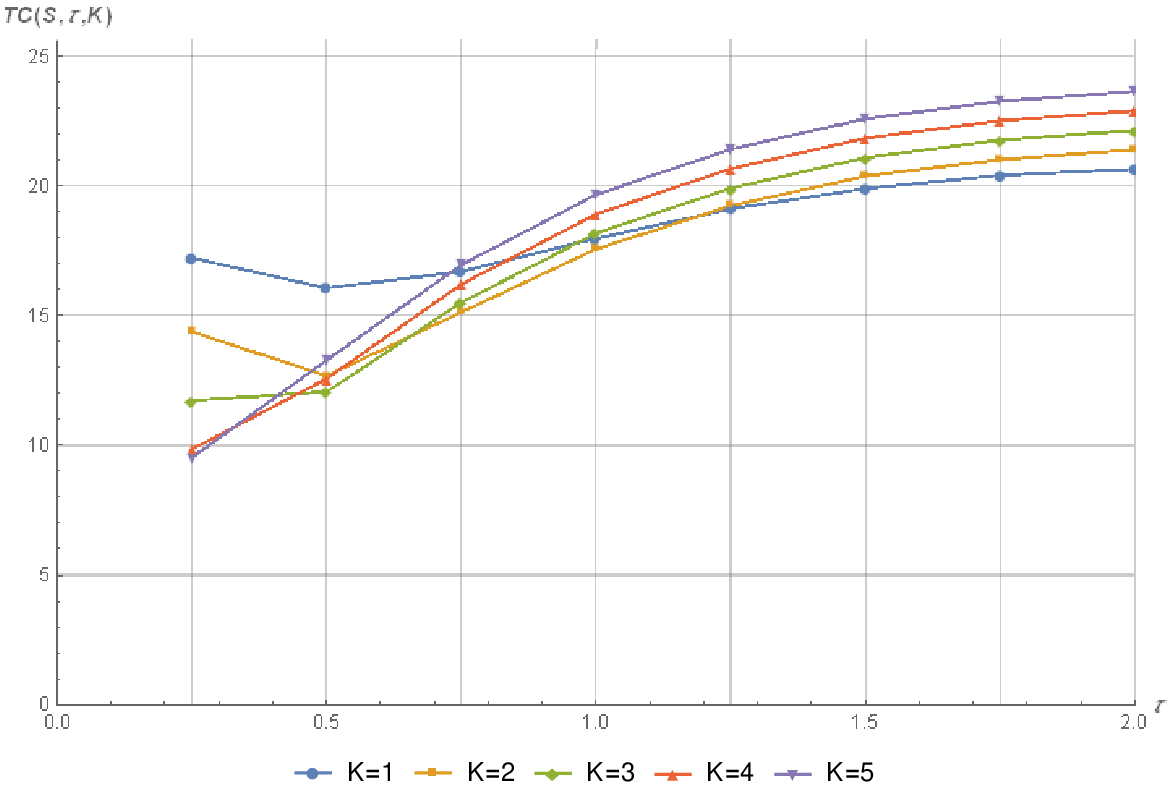, width=8.5cm, height=6cm}
  \caption{Total cost variation as a function of $\tau$ and the capacity $K$, $C_u=20$, $C_d=20$, $C_w=0.75$.}\label{fig6}
 \end{minipage} \hfill
 \begin{minipage}[c]{0.48\linewidth}
 \label{fig:SZnetwork}
  \centering\epsfig{figure=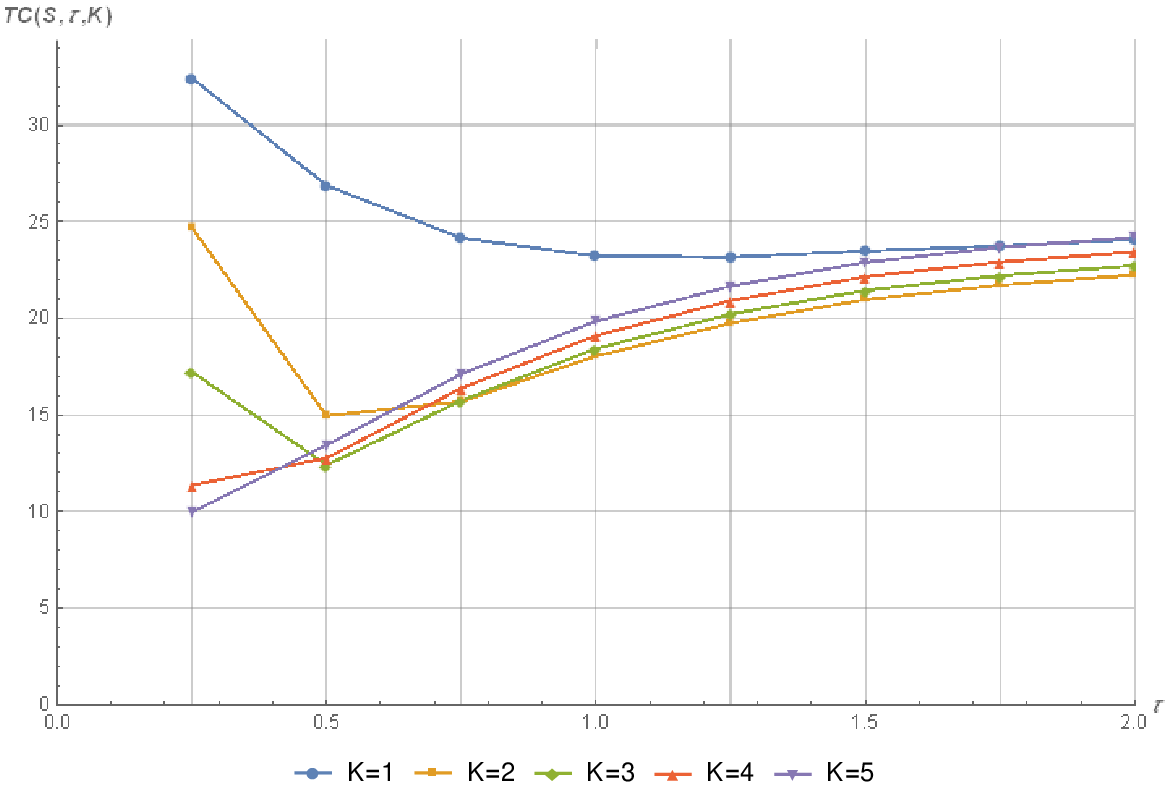, width=8.5cm, height=6cm}
  \caption{Total cost variation as a function of $\tau$ and the capacity $K$, $C_u=20$, $C_d=40$, $C_w=0.75$.}\label{fig7}
 \end{minipage}
\end{figure}

From \ref{fig4}-\ref{fig7} we observe that the optimal costs can vary quite substantially in the capacity of the repair shop $K$, especially when the preventive maintenance threshold is set low. For example, Figures \ref{fig4} and \ref{fig5} indicate that the optimal costs increase significantly when $K$ decreases from $5$ to $1$ for low values of $\tau$.
Figures \ref{fig4}-\ref{fig7} furthermore show that for low capacity $K$, the optimal policy prescribes a high threshold limit to balance the cost of an overloaded repair shop. Conversely, when the repair shop has more capacity, the optimal policy prescribes a lower age-threshold, allowing the asset to take advantage of preventive maintenance and reduce downtime due to breakdowns. This is possible because the repair shop is less congested, allowing parts to be repaired in less time than with a similar system of lower capacity. From these figures we can also deduce that the cost savings due to preventive maintenance can be quite significant. Indeed, everything else fixed, we see that decreasing $\tau$ can lead to significant decreases in costs, which implies that coordinated preventive maintenance (i.e. low $\tau$) has benefit over coordinated corrective (i.e. high $\tau$).
 
In Figures \ref{fig8}-\ref{fig11} we plot the optimal value of inventory level $S$ as a function of $\tau$ and the capacity $K$. We vary $\tau$ between 0 and 2 and $K$ between 1 and 5. We consider the values $C_u \in \{10, 20\}$, $C_d \in \{20, 40\}$ and $C_w = 0.75$.

\begin{figure} [h]
 \begin{minipage}[c]{0.48\textwidth}
 \label{fig:realnework}
  \centering\epsfig{figure=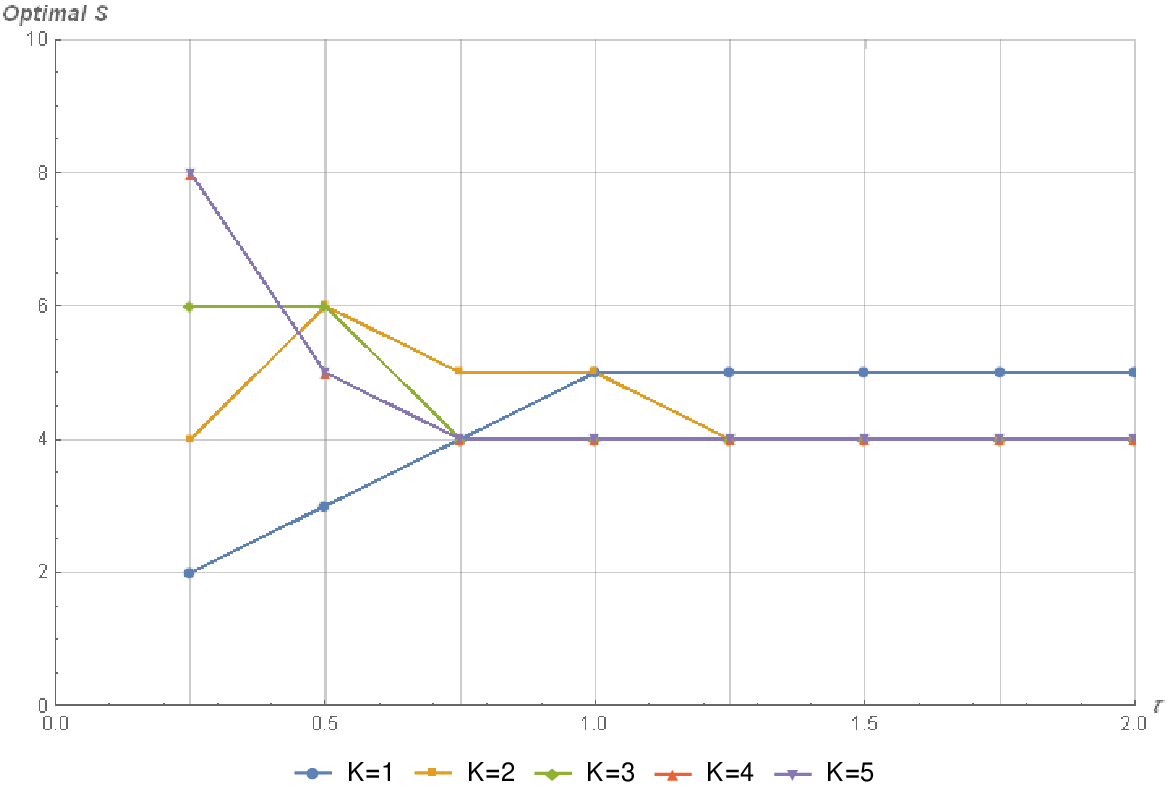, width=8.5cm, height=6cm}
  \caption{Optimal stock level $S$ variation as a function of $\tau$ and the capacity $K$, $C_u=10$, $C_d=20$, $C_w=0.75$.}\label{fig8}
 \end{minipage} \hfill
 \begin{minipage}[c]{0.48\linewidth}
 \label{fig:SZnetwork}
  \centering\epsfig{figure=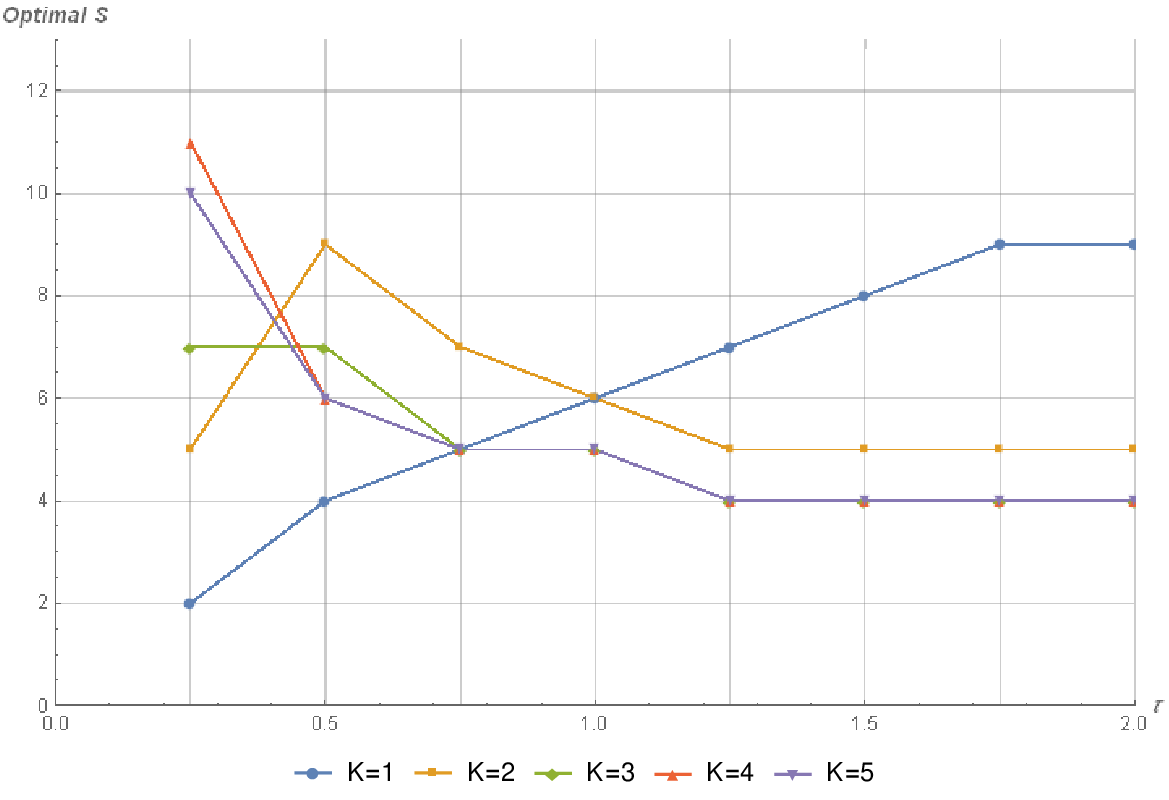, width=8.5cm, height=6cm}
  \caption{Optimal stock level $S$ variation as a function of $\tau$ and the capacity $K$, $C_u=10$, $C_d=40$, $C_w=0.75$.}\label{fig9}
 \end{minipage}
\end{figure}

\begin{figure}[h]
 \begin{minipage}[c]{0.48\textwidth}
 \label{fig:realnework}
  \centering\epsfig{figure=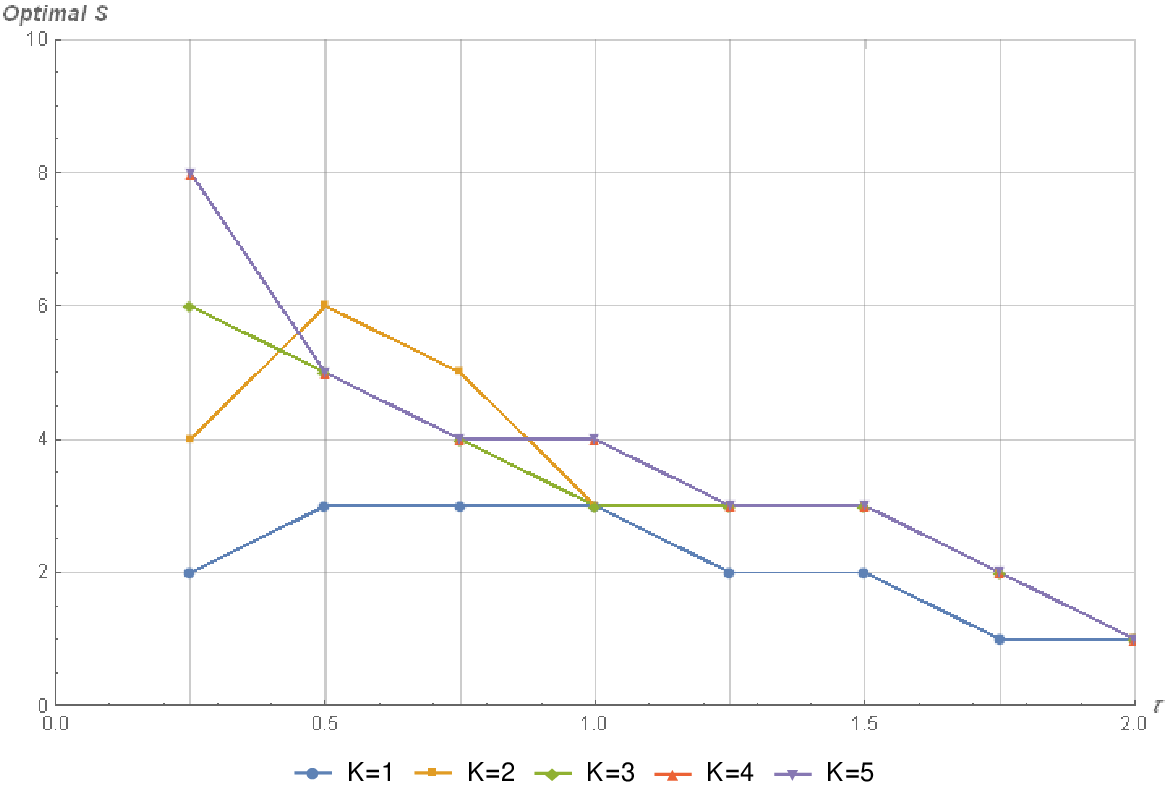, width=8.5cm, height=6cm}
  \caption{Optimal stock level $S$ variation as a function of $\tau$ and the capacity $K$, $C_u=20$, $C_d=20$, $C_w=0.75$.}\label{fig10}
 \end{minipage} \hfill
 \begin{minipage}[c]{0.48\linewidth}
 \label{fig:SZnetwork}
  \centering\epsfig{figure=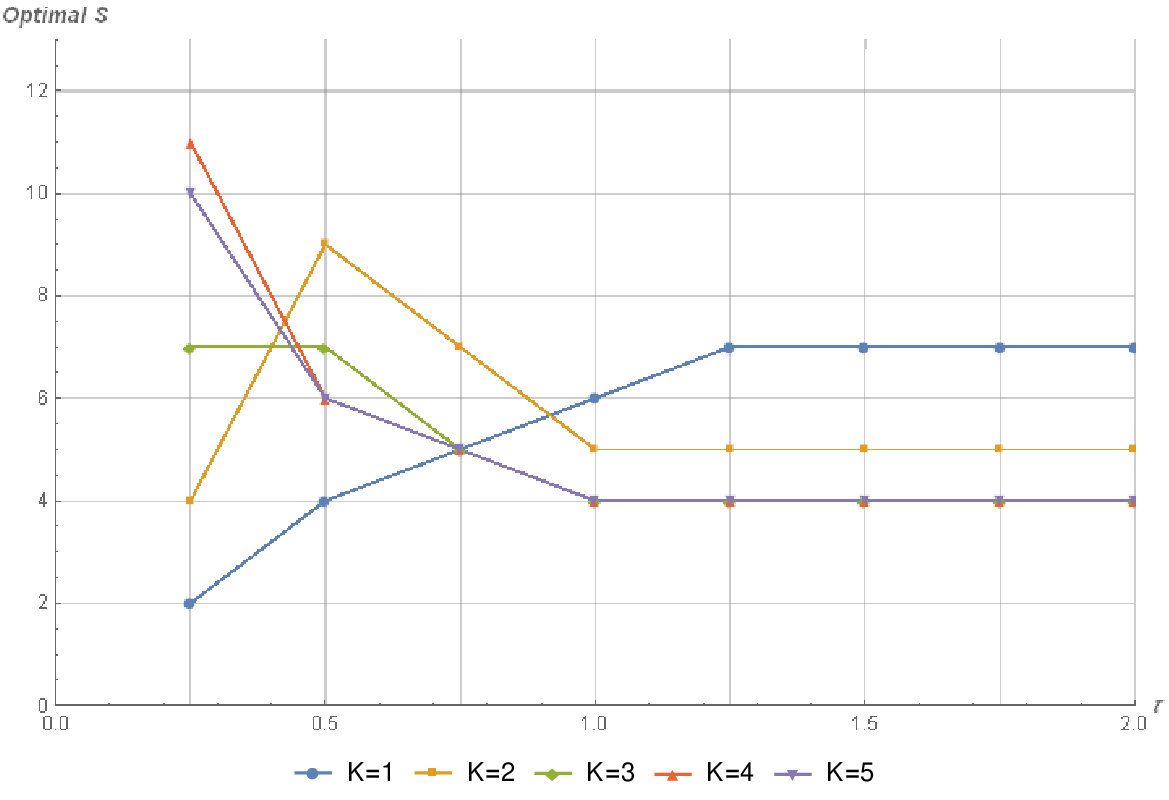, width=8.5cm, height=6cm}
  \caption{Optimal stock level $S$ variation as a function of $\tau$ and the capacity $K$, $C_u=20$, $C_d=40$, $C_w=0.75$.}\label{fig11}
 \end{minipage}
\end{figure}

We observe that the optimal inventory level $S$ is rather small when the repair shop has low capacity and a low age-threshold. If the capacity $K$ increases, we see that the inventory level $S$ also increases. However, as the age-threshold $\tau$ increases, we see that the inventory level $S$ reaches a fixed value that becomes less sensitive to the capacity $K$ of the repair shop. The optimal policy is therefore to choose the value of $S$ that minimizes the total cost for a given value of $K$ with less dependence on the value of the age-threshold. Furthermore, as $K$ increases, the optimal policy prescribes a lower inventory level $S$ since there is less need to compensate for the congestion time when items are in repair. 

Our numerical investigation has so far focused on the effect of varying parameters on the optimal policy and the optimal cost. 
We did so under the assumption that the underlying lifetime distribution is correct. However, in practice, there may be misspecfication in this lifetime distribution as it is often estimated from historical data. Fortunately, as our model permits a general lifetime distribution, we can quantify how robust our model is with respect to such misspecification. To do this, we proceed as follows. We will compare two capital goods with two estimates of the lifetime distribution. In the first case, we assume that the true litetime distribution is a gamma distribution, but that the decision-maker wrongly assumes that it is a Weibull distribution. 
In the second case, we do the opposite, i.e., the true lifetime distribution is a Weibull distribution while the decision maker assumes it is gamma distributed. Both lifetime distributions admit the same mean and the same coefficient of variation. We are now interested in the impact of this error on the total cost. That is, we compute the cost of the optimal policy obtained under the wrong lifetime distribution when implemented in the system with the true lifetime distribution, and then compare this cost with the optimal cost for that latter system.  

Our test bed consists of 3200 instances by permuting various values for $C_u$, $C_d$, $1/\mu_r$, $K$, as well as the coefficient of variation $C_v$ of the true lifetime distribution and the estimated one. For both distributions the average lifetime is fixed to 2 time units. Table \ref{testbed} contains all parameter values. In Table \ref{GammaVsWeibull} we provide the average of the observed percentage cost difference as well as the maximum percentage difference over all instances for a specific parameter value.  The true lifetime distribution can be found in the first column of the Table \ref{GammaVsWeibull} (i.e. Weibull or Gamma); the lifetime distribution as assumed by the decision maker will then be the other.
\begin{figure}[ht]
\begin{center}
\psfig{file=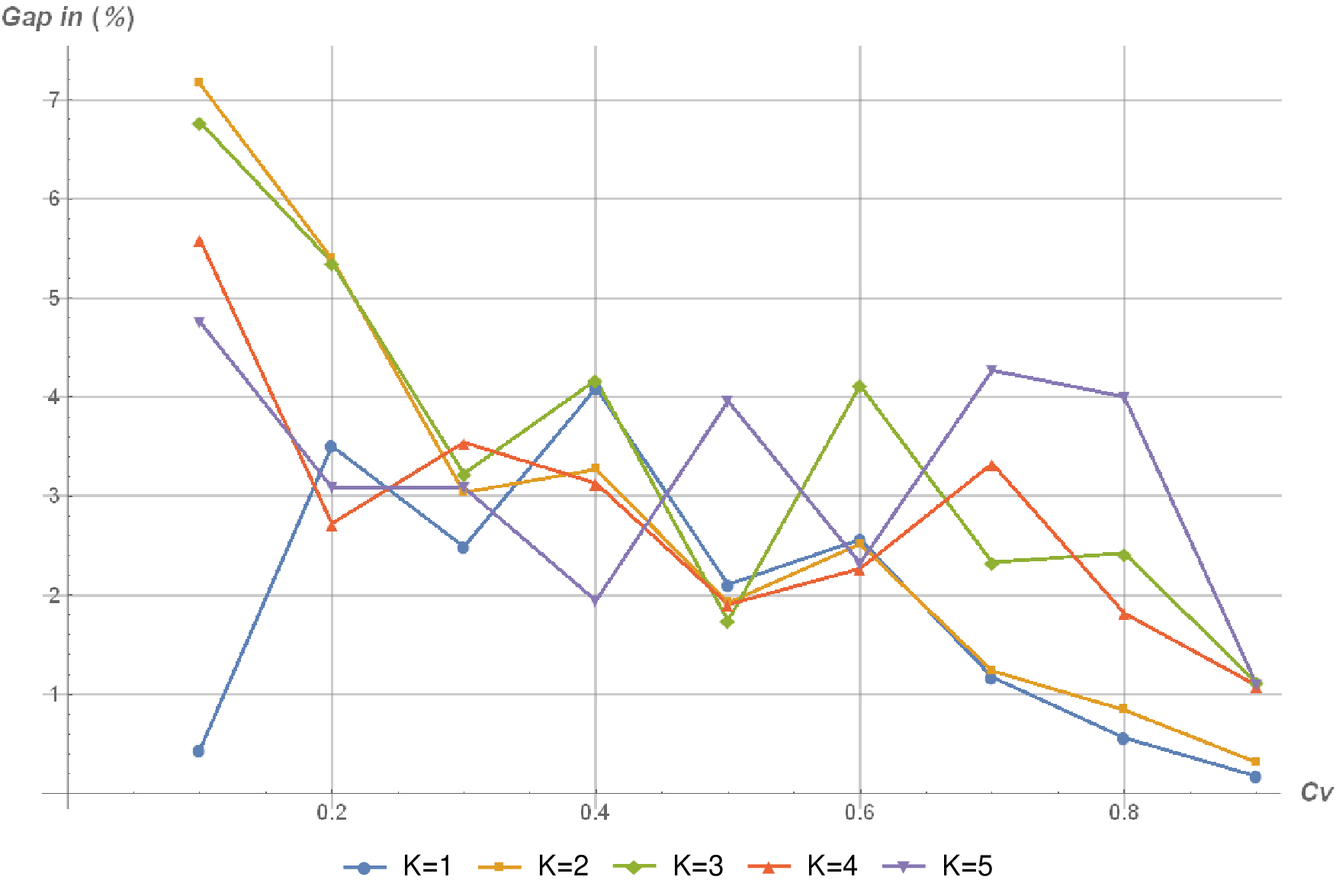, width=10cm, height=5cm}
\caption{Example: Maximum percentage difference between total costs when the lifetime is Weibull versus Gamma distribution}\label{fig12}
\end{center}
\end{figure}

\begin{table}[htbp]
   \begin{center}
{\footnotesize
\renewcommand{\arraystretch}{1}
\setlength{\tabcolsep}{0.28cm}
  \caption{Test bed}\label{testbed}
    \scalebox{0.75}{%
    \begin{tabular}{|c|c|}
    \hline
   Parameters& Values \\
   \hline
$C_d$& $\{5,10,20,40\}$ \\
\hline
$C_u$& $\{1,5,10,20\}$ \\
\hline
$C_v$& $\{0.1,0.2,0.3,0.4,0.5,0.6,0.7,0.8,0.9\}$ \\
\hline
$1/\mu_r$& $\{1,2,3,4\}$ \\
\hline$K$& $\{1,2,3,4,5\}$ \\
  \hline
 \end{tabular}}}
\end{center}
\end{table}
On average over all instances, we observe that the percentage difference when using a wrong estimation of the true lifetime distribution does not have much impact on the total cost. While the maximum difference can go up to 8\%, on average it is only 0.23\% (0.21\%) when the true lifetime distribution is a gamma (Weibull) distribution while the decision-maker assumes that it is Weibull (gamma) distributed. The difference increases with the value of $C_u$ and $K$ and decreases with the value of $C_d$, $C_v$ and $1/\mu_r$. Particular attention should be paid to these cases, especially when $C_d$, $C_v$ and $1/\mu_r$ are low because the system then requires different threshold parameters $\tau$ and a stock level $S$ than those found using the real lifetime distribution. The good performance over all instances is due to the role of the threshold $\tau$ which offsets in most cases the impact in total cost when using a wrong lifetime distribution instead of the real one.

\subsection{Multiple capital goods}
While we developed an exact approach for the single capital good setting, we relied on an approximative procedure for the multiple capital goods setting. Our first objective is therefore to quantify how accurate our approximative procedure is. 
To that end, we consider a setting with two capital goods 1 and 2 whose failure follows an Erlang distribution with a mean $\mu_1=1$ and $\mu_2=2$ and a number of phases respectively equal to 3 and 6. The cost parameters are fixed as follows: $C_u \in\{10,20\}$, $C_d \in\{20,40\}$,  $C_a =0.25$ and $C_w =0.75$. We also set $\tau \in\{0.25,0.5,1,1.25,1.5,1.75,2\}$, $K\in\{1,2,3,4,5\}$ and $\mu_r =1$. For this setting, we assess the value of our approximative procedure by comparing our approximations with simulated values, both for the first come first served and the priority repair shop discipline.  
The simulation model is built in Arena Rockwell. The simulation length is set to 100000 units, which guarantees to obtain a response within an accuracy of $10^{-3}$. Tables \ref{tb1}-\ref{tb3} provides the results of the comparison between our approximative procedure and simulation. 
Based on these tables we conclude that our approximations are quite accurate: the difference between the costs of our model and the costs obtained through simulation is less than 2\% for all instances considered. This shows that our approximative procedure that decomposes the multiple capital goods setting into single capital goods for which we use our exact results performs well.

Now that we found that our approximations are accurate, the second objective of our numerical investigation is to study the effect of using a priority discipline in the repair shop on the overall performance. 
The presence of a scheduling policy based on criticality and costs allows to reduce the downtime and the dedicated stock levels for the most critical capital assets (i.e. those with highest costs).
We study two capital goods, and give priority to capital good 1 over capital good 2. We do not add any additional cost due to the priority to the repair shop. Since capital good 1 has higher priority, we set the capital cost of capital good 1 higher than capital good 2. The results of the comparison are presented in Table \ref{tb4} and \ref{PrioVSNonPrio}.

As the cost of downtime for the priority captial good is high and the cost of failure before threshold is low, and for the non-priority capital good the cost of downtime is low and the cost of failure is high, the potential savings from prioritizing capital good 1 over capital good 2 can reach up to 19\% over first comes first served. This is especially true when the age threshold of capital good 2 decreases. Indeed, this leads to more repair shop utilization for capital good 2, which would delay repair for capital good 1 if served under a first comes first served discipline.
We also find that the age threshold for capital good 1 has less impact on the performance of the prioritized system than the non-prioritized system. Indeed, priority plays somewhat the role of age-threshold to reduce the waiting time and thus the total cost of the system.

\begin{figure}[ht]
\begin{center}
\psfig{file=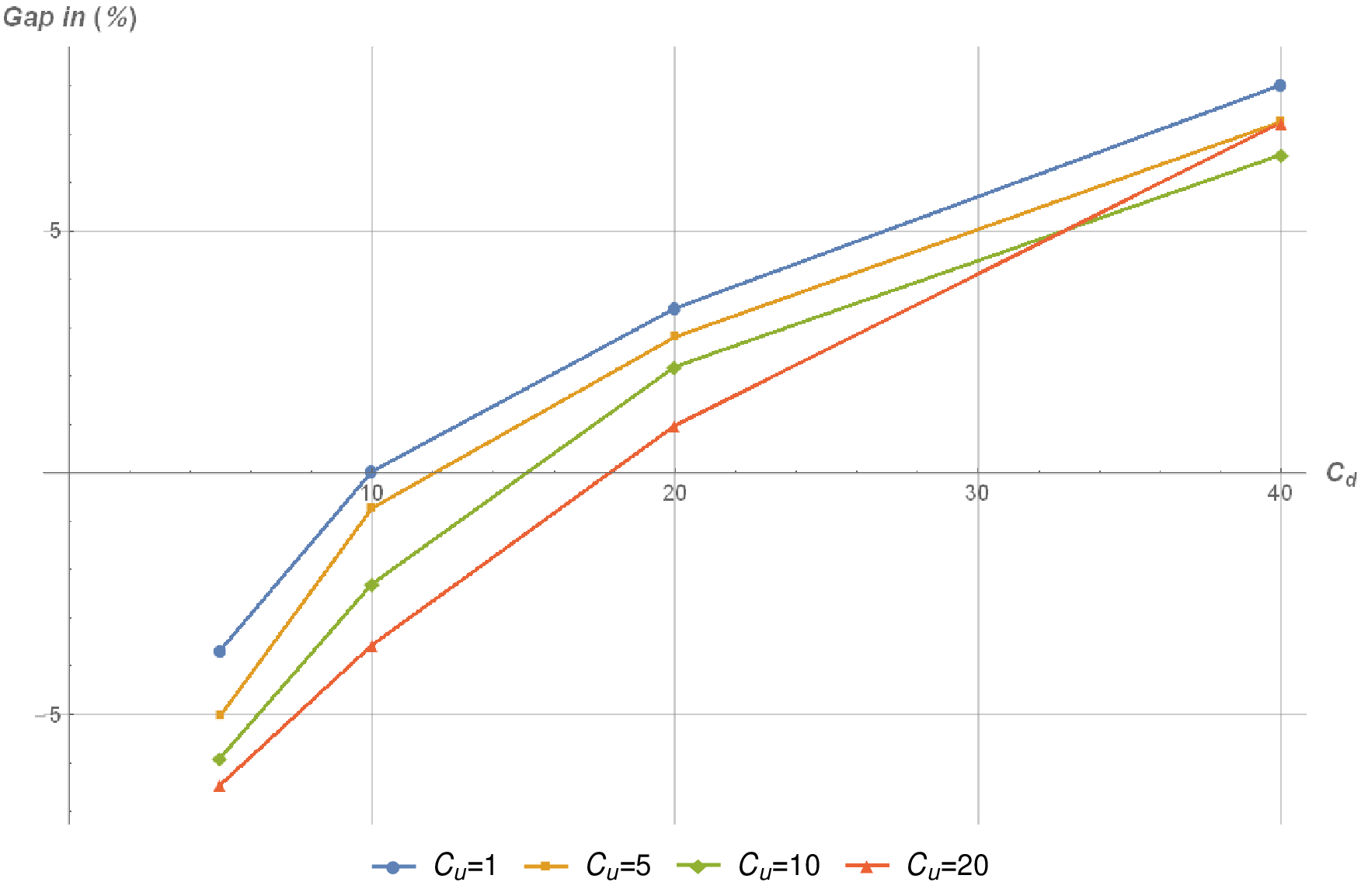, width=10cm, height=5cm}
\caption{Example: Priority versus non priority: $C_u=20$ and $C_d=40$ for the capital good with the highest priority}\label{fig13}
\end{center}
\end{figure}

\section{Conclusions}
\label{sec:concl}
We have considered a repairable spare parts inventory system consisting of a stock point for a single and multiple capital goods and a capacitated repair shop. Parts are replaced either when there is a failure or when their lifetime reaches an age threshold. We have modelled the system as as a cyclic tandem queue network where the number of customers is fixed and equal to base stock level. In the case of a single capital good, we have derived a closed form expression of the steady-state probability of the number of orders in the queuing system. After deriving some convexity properties of the steady state probabilities and the cost function, we have proposed an algorithm to optimize the number of repairables in the system, the age-threshold, the capacity of the repair shop and the base stock level. In the case of multiple capital goods, since a closed form expression of the steady state probability cannot be derived, we provide an approximate expression by replacing in the queuing network the general service time queues with state-dependent exponential service time queues where the latter is obtained by analyzing each station with a general service time in isolation but with a state-dependent arrival rate. We have also considered multiple priorities in the repair shop and a preemptive service discipline, which means that when a failed part with higher priority arrives at the repair shop and another part with a lower priority is being repaired, the server immediately stops repairing the latter and starts with the former. 

The numerical investigation reveals that the cost savings due to preventive maintenance can be substantial. We have shown that in the system with single capital good, when the capacity of the workshop increases, a lower age-threshold should be used, which allows the asset to take advantage of preventive maintenance and hence reduce costs. The results also show that the system uses a low base stock level when the repair shop has a low capacity and a low age-threshold. This base stock level increases with the capacity of the repair shop until it reaches a fixed value that becomes less sensitive to the capacity. Another finding of the numerical analysis is the robustness of our model if there is a misspecification  of the lifetime distribution instead of the general distribution. Moreover, in the case of multiple capital goods, we show the high quality of the derived approximation after comparing it to that obtained through a simulation model and we also show the performance improvement that results from using the priority discipline in the repair shop. 

It is important to notice that the analysis developed in this paper assumes exponential service times in the repair shop. Hence, it would be interesting to extend this work by considering generally distributed service times in the repair shop in addition to the generally distributed lifetime of the parts. Another restrictive assumption made in the case of multiple capital goods is the consideration of a single repairmen in the repair shop, which facilitates the priority rule in the system. Therefore, an interesting avenue for further research would be to consider the case of multiple repairmen, which is complex to analyze even in the case of exponentially distributed service times due to the complexity related to the consideration of a scheduling policy for the parts in the repair shop.

\appendix
\section{Gamma versus Weibull distribution}\label{Append-1}
\begin{table}[H]
 \begin{center}
{\footnotesize
\renewcommand{\arraystretch}{0.5}
\setlength{\tabcolsep}{0.15cm}
  \caption{Percentage difference between total costs when the lifetime is Gamma versus Weibull distribution.}\label{GammaVsWeibull}
  \scalebox{0.75}{%
       \begin{tabular}{|c|l|ccccc|ccccc|}
    \hline
    \multirow{2}[4]{*}{Distribution} & \multicolumn{1}{c|}{Parameters} & \multicolumn{5}{c|}{Average}          & \multicolumn{5}{c|}{Maximum} \\
\cline{3-12}          & \multicolumn{1}{c|}{setting} & \multicolumn{1}{c}{$K=1$} & \multicolumn{1}{c}{$K=2$} & \multicolumn{1}{c}{$K=3$} & \multicolumn{1}{c}{$K=4$} & \multicolumn{1}{c|}{$K=5$} & \multicolumn{1}{c}{$K=1$} & \multicolumn{1}{c}{$K=2$} & \multicolumn{1}{c}{$K=3$} & \multicolumn{1}{c}{$K=4$} & \multicolumn{1}{c|}{$K=5$} \\
    \hline
    \multirow{5}[4]{*}{Gamma} & \multicolumn{1}{l|}{$C_d$} & 0   & 0   & 0   & 0   & 0 & 4   & 7   & 7   & 8   & 6 \\
\cline{2-12}          & \multicolumn{1}{r|}{5} & 0   & 0   & 0   & 0   & 0 & 2   & 7   & 7   & 8   & 5 \\
          & \multicolumn{1}{r|}{10} & 0   & 0   & 0   & 0   & 0 & 2   & 3   & 4   & 6   & 6 \\
          & \multicolumn{1}{r|}{20} & 0   & 0   & 0   & 0   & 0 & 3   & 5   & 5   & 7   & 5 \\
          & \multicolumn{1}{r|}{40} & 0   & 0   & 0   & 0   & 0 & 4   & 4   & 5   & 7   & 5 \\
    \hline
    \multirow{5}[4]{*}{Weibull} & \multicolumn{1}{l|}{$C_d$} & 0   & 0   & 0   & 0   & 0& 4   & 7   & 7   & 6   & 5 \\
\cline{2-12}          & \multicolumn{1}{r|}{5} & 0   & 0   & 0   & 0   & 0 & 2   & 4   & 5   & 5   & 4 \\
          & \multicolumn{1}{r|}{10} & 0   & 0   & 0   & 0   & 0 & 3   & 7   & 7   & 6   & 5 \\
          & \multicolumn{1}{r|}{20} & 0   & 0   & 0   & 0   & 0 & 4   & 5   & 5   & 4   & 3 \\
          & \multicolumn{1}{r|}{40} & 0   & 0   & 0   & 0   & 0 & 4   & 5   & 4   & 4   & 4 \\
    \hline
    \multirow{5}[4]{*}{Gamma} & \multicolumn{1}{l|}{$C_u$} & 0   & 0   & 0   & 0   & 0& 4   & 7   & 7   & 8   & 6 \\
\cline{2-12}          & \multicolumn{1}{r|}{1} & 0   & 0   & 0   & 0   & 0 & 0   & 1   & 1   & 1   & 1 \\
          & \multicolumn{1}{r|}{5} & 0   & 0   & 0   & 0   &0 & 2   & 2   & 3   & 2   & 2 \\
          & \multicolumn{1}{r|}{10} & 0   & 0   & 0   & 0   & 0 & 4   & 3   & 3   & 3   & 3 \\
          & \multicolumn{1}{r|}{20} & 0   & 0   & 1   & 1   &1 & 3   & 7   & 7   & 8   & 6 \\
    \hline
    \multirow{5}[4]{*}{Weibull} & \multicolumn{1}{l|}{$C_u$} & 0   & 0   & 0   & 0   & 0 & 4   & 7   & 7   & 6   & 5 \\
\cline{2-12}          & \multicolumn{1}{r|}{1} & 0   & 0   & 0   & 0   & 0& 0   & 1   & 1   & 1   & 1 \\
          & \multicolumn{1}{r|}{5} & 0   & 0   & 0   & 0   & 0 & 3   & 2   & 2   & 2   & 2 \\
          & \multicolumn{1}{r|}{10} & 0   & 0   & 0   & 0   & 0 & 4   & 5   & 4   & 3   & 3 \\
          & \multicolumn{1}{r|}{20} & 0   & 0   & 1   & 0   & 1 & 4   & 7   & 7   & 6   & 5 \\
    \hline
    \multirow{10}[4]{*}{Gamma} & \multicolumn{1}{l|}{$Cv$} & 0   & 0   & 0   & 0   & 0& 4   & 7   & 7   & 8   & 6 \\
\cline{2-12}          & \multicolumn{1}{r|}{0.1} & 0   & 0   & 0   & 0   & 1 & 3   & 5   & 5   & 4   & 4 \\
          & \multicolumn{1}{r|}{0.2} & 0   & 1   & 0   & 0   & 0& 2   & 7   & 5   & 3   & 4 \\
          & \multicolumn{1}{r|}{0.4} & 0   & 0   & 0   & 0   & 0 & 1   & 5   & 2   & 4   & 2 \\
          & \multicolumn{1}{r|}{0.6} & 0   & 0   & 0   & 0   & 0 & 2   & 3   & 3   & 3   & 4 \\
          & \multicolumn{1}{r|}{0.8} & 0   & 0   & 0   & 0   & 0 & 1   & 2   & 2   & 2   & 4 \\
    \hline
    \multirow{10}[4]{*}{Weibull} & \multicolumn{1}{l|}{$Cv$} & 0   & 0   & 0   & 0   & 0 & 4   & 7   & 7   & 6   & 5 \\
\cline{2-12}          & \multicolumn{1}{r|}{0.1} & 0   & 0   & 1   & 1   & 1& 0   & 7   & 7   & 6   & 5 \\
          & \multicolumn{1}{r|}{0.2} & 0   & 1   & 0   & 0   & 0& 4   & 5   & 5   & 3   & 3 \\
          & \multicolumn{1}{r|}{0.4} & 0   & 0   & 0   & 0   & 0& 4   & 3   & 4   & 3   & 2 \\
          & \multicolumn{1}{r|}{0.6} & 0   & 0   & 0   & 0   & 0 & 3   & 3   & 4   & 2   & 2 \\
          & \multicolumn{1}{r|}{0.8} & 0   & 0   & 0   & 0   & 0 & 1   & 1   & 2   & 2   & 4 \\
    \hline
    \multirow{5}[4]{*}{Gamma} & \multicolumn{1}{l|}{$L$}      & 0   & 0   & 0   & 0   & 0   & 4   & 7   & 7   & 8   & 6 \\
\cline{2-12}          & \multicolumn{1}{r|}{1}      & 0   & 0   & 0   & 0   & 0   & 4   & 5   & 7   & 8   & 6 \\
          &\multicolumn{1}{r|}{2}      & 0   & 0   & 0   & 0   & 0   & 2   & 7   & 5   & 4   & 4 \\
          & \multicolumn{1}{r|}{3}     & 0   & 0   & 0   & 0   & 0   & 3   & 3   & 2   & 4   & 3 \\
          & \multicolumn{1}{r|}{4}      & 0   & 0   & 0   & 0   & 0   & 1   & 2   & 3   & 3   & 4 \\
    \hline
    \multirow{5}[4]{*}{Weibull} & \multicolumn{1}{l|}{$L$}      & 0   & 0   & 0   & 0   & 0   & 4   & 7   & 7   & 6   & 5 \\
\cline{2-12}          & \multicolumn{1}{r|}{1}      & 0   & 0   & 0   & 0   & 0   & 4   & 7   & 7   & 6   & 5 \\
          & \multicolumn{1}{r|}{2}      & 0   & 0   & 0   & 0   & 0   & 4   & 3   & 5   & 5   & 4 \\
          & \multicolumn{1}{r|}{3}      & 0   & 0   & 0   & 0   & 0   & 1   & 5   & 4   & 2   & 3 \\
          & \multicolumn{1}{r|}{4}     & 0   & 0   & 0   & 0   & 0   & 1   & 3   & 3   & 3   & 3 \\
    \hline
    \end{tabular}}}
\end{center}
\end{table}

\section{Comparison with Simulation}\label{Append-2}
\begin{table}[H]
\begin{center}
{\footnotesize
\renewcommand{\arraystretch}{0.6}
\setlength{\tabcolsep}{0.15cm}
  \caption{Comparison with simulation under a fixed capacity $K=1$ and First come first served repair shop discipline.}\label{tb1}
  \scalebox{0.75}{%
    \begin{tabular}{cccc|ccccccc|ccccc|c}
     \hline
    \multicolumn{4}{c|}{\textbf{Input parameters}} & \multicolumn{7}{c|}{\textbf{Our model}} & \multicolumn{5}{c|}{\textbf{Simulation}}&\multicolumn{1}{c}{Difference} \\   
      \cline{1-16}
$C^{2}_{u}$&	$C^{2}_{d}$&	 $\tau_1$ &	 $\tau_2$& 	$S_1$& 	$S_2$& 	$P_1(0)$& $E[I_1]$ & $P_2(0)$& $E[I_2]$& $TC(S,\tau,K)$ & $P_1(0)$& $E[I_1]$ & $P_2(0)$& $E[I_2]$& $TC(S,\tau,K)$& in (\%)\\
    \hline
10	&	20	&	0.5	&	0.5	&	4	&	1	&	0.63	&	0.48	&	0.9	&	0.1	&	47.79	&	0.63	&	0.47	&	0.9	&	0.1	&	47.78	&	0	\\
10	&	20	&	1	&	0.5	&	6	&	1	&	0.35	&	1.2	&	0.92	&	0.08	&	44.31	&	0.36	&	1.16	&	0.92	&	0.08	&	44.39	&	0.2	\\
10	&	20	&	2	&	0.5	&	7	&	1	&	0.21	&	2.12	&	0.91	&	0.09	&	44.29	&	0.2	&	2.08	&	0.91	&	0.09	&	44.22	&	-0.2	\\
10	&	20	&	0.5	&	1	&	1	&	4	&	0.88	&	0.12	&	0.28	&	1.27	&	44.11	&	0.88	&	0.13	&	0.28	&	1.23	&	44.04	&	-0.2	\\
10	&	20	&	0.5	&	2	&	2	&	6	&	0.8	&	0.22	&	0.05	&	3.13	&	40.24	&	0.8	&	0.21	&	0.04	&	3.09	&	40.1	&	-0.4	\\
10	&	20	&	1	&	1	&	4	&	1	&	0.39	&	0.98	&	0.79	&	0.21	&	42.56	&	0.4	&	0.96	&	0.79	&	0.21	&	42.53	&	-0.1	\\
10	&	20	&	1	&	2	&	2	&	5	&	0.67	&	0.39	&	0.07	&	2.57	&	38.37	&	0.66	&	0.38	&	0.07	&	2.5	&	38.24	&	-0.3	\\
10	&	20	&	2	&	1	&	4	&	2	&	0.38	&	1.05	&	0.65	&	0.4	&	42.37	&	0.38	&	1	&	0.65	&	0.39	&	42.38	&	0	\\
10	&	20	&	2	&	2	&	2	&	5	&	0.6	&	0.49	&	0.07	&	2.67	&	38.27	&	0.59	&	0.48	&	0.06	&	2.59	&	38.15	&	-0.3	\\
20	&	20	&	0.5	&	0.5	&	4	&	1	&	0.63	&	0.48	&	0.9	&	0.1	&	47.8	&	0.63	&	0.47	&	0.9	&	0.1	&	47.79	&	0	\\
20	&	20	&	1	&	0.5	&	6	&	1	&	0.35	&	1.2	&	0.92	&	0.08	&	44.32	&	0.36	&	1.16	&	0.92	&	0.08	&	44.4	&	0.2	\\
20	&	20	&	2	&	0.5	&	7	&	1	&	0.21	&	2.12	&	0.91	&	0.09	&	44.3	&	0.2	&	2.08	&	0.91	&	0.09	&	44.23	&	-0.2	\\
20	&	20	&	0.5	&	1	&	1	&	4	&	0.88	&	0.12	&	0.28	&	1.27	&	44.72	&	0.88	&	0.13	&	0.28	&	1.23	&	44.65	&	-0.1	\\
20	&	20	&	0.5	&	2	&	2	&	4	&	0.77	&	0.25	&	0.14	&	1.82	&	43.27	&	0.78	&	0.24	&	0.13	&	1.74	&	43.19	&	-0.2	\\
20	&	20	&	1	&	1	&	4	&	1	&	0.39	&	0.98	&	0.79	&	0.21	&	42.74	&	0.4	&	0.96	&	0.79	&	0.21	&	42.71	&	-0.1	\\
20	&	20	&	1	&	2	&	3	&	3	&	0.56	&	0.58	&	0.29	&	1.1	&	41.29	&	0.56	&	0.56	&	0.29	&	1.03	&	41.33	&	0.1	\\
20	&	20	&	2	&	1	&	4	&	2	&	0.38	&	1.05	&	0.65	&	0.4	&	42.67	&	0.38	&	1	&	0.65	&	0.39	&	42.68	&	0	\\
20	&	20	&	2	&	2	&	3	&	4	&	0.51	&	0.68	&	0.18	&	1.62	&	41.15	&	0.51	&	0.65	&	0.18	&	1.51	&	41.14	&	0	\\
   \hline
    \multicolumn{16}{c}{Fixed parameters: $C^{1}_{u}=20$, $C^{1}_d=40$,$C^{1}_a=C^{1}_a=0.25$,$C_w=0.5$.}\\
    \hline
    \end{tabular}}}
\end{center}
\end{table}

\begin{table}[H]
\begin{center}
{\footnotesize
\renewcommand{\arraystretch}{0.6}
\setlength{\tabcolsep}{0.15cm}
  \caption{Comparison with simulation under a fixed capacity $K=5$ and First come first served repair shop discipline.}\label{tb2}
  \scalebox{0.75}{%
    \begin{tabular}{cccc|ccccccc|ccccc|c}
     \hline
    \multicolumn{4}{c|}{\textbf{Input parameters}} & \multicolumn{7}{c|}{\textbf{Our model}} & \multicolumn{5}{c|}{\textbf{Simulation}}&\multicolumn{1}{c}{Difference} \\   
      \cline{1-16}
$C^{2}_{u}$&	$C^{2}_{d}$&	 $\tau_1$ &	 $\tau_2$& 	$S_1$& 	$S_2$& 	$P_1(0)$& $E[I_1]$ & $P_2(0)$& $E[I_2]$& $TC(S,\tau,K)$ & $P_1(0)$& $E[I_1]$ & $P_2(0)$& $E[I_2]$& $TC(S,\tau,K)$& in (\%)\\
    \hline
10	&	20	&	0.5	&	0.5	&	7	&	6	&	0	&	4.51	&	0.01	&	3.7	&	14.33	&	0.01	&	4.44	&	0.01	&	3.61	&	14.29	&	0	\\
10	&	20	&	1	&	0.5	&	5	&	5	&	0	&	3.55	&	0.01	&	2.96	&	20.3	&	0	&	3.66	&	0.01	&	2.94	&	20.14	&	0	\\
10	&	20	&	2	&	0.5	&	4	&	5	&	0.01	&	2.73	&	0.01	&	2.97	&	24.59	&	0	&	2.96	&	0.01	&	2.96	&	24.35	&	0	\\
10	&	20	&	0.5	&	1	&	6	&	4	&	0	&	3.8	&	0	&	2.95	&	14.16	&	0	&	3.82	&	0	&	2.95	&	14.08	&	0	\\
10	&	20	&	0.5	&	2	&	6	&	3	&	0	&	3.83	&	0	&	2.37	&	16.33	&	0	&	3.85	&	0	&	2.41	&	16.25	&	0	\\
10	&	20	&	1	&	1	&	4	&	4	&	0.01	&	2.66	&	0	&	2.98	&	20.49	&	0.01	&	2.72	&	0	&	2.98	&	20.35	&	0	\\
10	&	20	&	1	&	2	&	4	&	3	&	0.01	&	2.66	&	0	&	2.37	&	22.69	&	0.01	&	2.71	&	0	&	2.41	&	22.57	&	0	\\
10	&	20	&	2	&	1	&	4	&	4	&	0.01	&	2.76	&	0	&	2.98	&	24.82	&	0	&	2.98	&	0	&	2.97	&	24.7	&	0	\\
10	&	20	&	2	&	2	&	4	&	3	&	0.01	&	2.76	&	0	&	2.37	&	27.02	&	0	&	2.98	&	0	&	2.4	&	26.9	&	0	\\
20	&	20	&	0.5	&	0.5	&	7	&	6	&	0	&	4.51	&	0.01	&	3.7	&	14.41	&	0.01	&	4.44	&	0.01	&	3.61	&	14.38	&	0	\\
20	&	20	&	1	&	0.5	&	5	&	5	&	0	&	3.55	&	0.01	&	2.96	&	20.38	&	0	&	3.66	&	0.01	&	2.94	&	20.23	&	0	\\
20	&	20	&	2	&	0.5	&	4	&	5	&	0.01	&	2.73	&	0.01	&	2.97	&	24.67	&	0	&	2.96	&	0.01	&	2.96	&	24.44	&	0	\\
20	&	20	&	0.5	&	1	&	6	&	4	&	0	&	3.8	&	0	&	2.95	&	15.01	&	0	&	3.82	&	0	&	2.95	&	14.94	&	0	\\
20	&	20	&	0.5	&	2	&	6	&	3	&	0	&	3.83	&	0	&	2.37	&	19.62	&	0	&	3.85	&	0	&	2.41	&	19.55	&	0	\\
20	&	20	&	1	&	1	&	4	&	4	&	0.01	&	2.66	&	0	&	2.98	&	21.35	&	0.01	&	2.72	&	0	&	2.98	&	21.2	&	0	\\
20	&	20	&	1	&	2	&	4	&	3	&	0.01	&	2.66	&	0	&	2.37	&	25.99	&	0.01	&	2.71	&	0	&	2.41	&	25.87	&	0	\\
20	&	20	&	2	&	1	&	4	&	4	&	0.01	&	2.76	&	0	&	2.98	&	25.67	&	0	&	2.98	&	0	&	2.97	&	25.56	&	0	\\
20	&	20	&	2	&	2	&	4	&	3	&	0.01	&	2.76	&	0	&	2.37	&	30.32	&	0	&	2.98	&	0	&	2.4	&	30.2	&	0	\\
   \hline
    \multicolumn{16}{c}{Fixed parameters: $C^{1}_{u}=20$, $C^{1}_d=40$,$C^{1}_a=C^{1}_a=0.25$,$C_w=0.5$.}\\
    \hline
    \end{tabular}}}
\end{center}
\end{table}
\begin{table}[H]
\begin{center}
{\footnotesize
\renewcommand{\arraystretch}{0.6}
\setlength{\tabcolsep}{0.15cm}
  \caption{Comparison with simulation under a fixed capacity $K=1$ and priority rule at the repair shop.}\label{tb3}
  \scalebox{0.75}{%
    \begin{tabular}{cccc|ccccccc|ccccc|c}
     \hline
    \multicolumn{4}{c|}{\textbf{Input parameters}} & \multicolumn{7}{c|}{\textbf{Our model}} & \multicolumn{5}{c|}{\textbf{Simulation}}&\multicolumn{1}{c}{Difference} \\   
      \cline{1-16}
$C^{2}_{u}$&	$C^{2}_{d}$&	 $\tau_1$ &	 $\tau_2$& 	$S_1$& 	$S_2$& 	$P_1(0)$& $E[I_1]$ & $P_2(0)$& $E[I_2]$& $TC(S,\tau,K)$ & $P_1(0)$& $E[I_1]$ & $P_2(0)$& $E[I_2]$& $TC(S,\tau,K)$& in (\%)\\
    \hline
10	&	20	&	0.5	&	0.5	&	3	&	1	&	0.54	&	0.62	&	0.99	&	0.01	&	46.69	&	0.54	&	0.62	&	0.99	&	0.01	&	46.64	&	0	\\
10	&	20	&	1	&	0.5	&	5	&	1	&	0.25	&	1.62	&	0.99	&	0.01	&	43.02	&	0.25	&	1.64	&	0.99	&	0.01	&	42.95	&	0	\\
10	&	20	&	2	&	0.5	&	7	&	1	&	0.11	&	3.2	&	0.98	&	0.02	&	43.72	&	0.1	&	3.35	&	0.97	&	0.03	&	43.36	&	0	\\
10	&	20	&	0.5	&	1	&	1	&	3	&	0.68	&	0.32	&	0.69	&	0.4	&	45.45	&	0.68	&	0.32	&	0.69	&	0.4	&	45.41	&	0	\\
10	&	20	&	0.5	&	2	&	1	&	4	&	0.68	&	0.32	&	0.48	&	0.85	&	42.96	&	0.68	&	0.32	&	0.48	&	0.86	&	42.79	&	0	\\
10	&	20	&	1	&	1	&	2	&	2	&	0.38	&	0.82	&	0.82	&	0.21	&	42.45	&	0.38	&	0.82	&	0.82	&	0.22	&	42.52	&	0	\\
10	&	20	&	1	&	2	&	1	&	5	&	0.56	&	0.44	&	0.31	&	1.52	&	39.48	&	0.56	&	0.44	&	0.3	&	1.54	&	39.35	&	0	\\
10	&	20	&	2	&	1	&	2	&	3	&	0.31	&	0.98	&	0.73	&	0.37	&	42.18	&	0.3	&	0.98	&	0.73	&	0.38	&	42.2	&	0	\\
10	&	20	&	2	&	2	&	1	&	6	&	0.51	&	0.49	&	0.22	&	2.11	&	39.08	&	0.51	&	0.49	&	0.22	&	2.14	&	38.95	&	0	\\
20	&	20	&	0.5	&	0.5	&	3	&	1	&	0.54	&	0.62	&	0.99	&	0.01	&	46.69	&	0.54	&	0.62	&	0.99	&	0.01	&	46.64	&	0	\\
20	&	20	&	1	&	0.5	&	5	&	1	&	0.25	&	1.62	&	0.99	&	0.01	&	43.02	&	0.25	&	1.64	&	0.99	&	0.01	&	42.95	&	0	\\
20	&	20	&	2	&	0.5	&	7	&	1	&	0.11	&	3.2	&	0.98	&	0.02	&	43.72	&	0.1	&	3.35	&	0.97	&	0.03	&	43.37	&	0	\\
20	&	20	&	0.5	&	1	&	1	&	3	&	0.68	&	0.32	&	0.69	&	0.4	&	45.71	&	0.68	&	0.32	&	0.69	&	0.4	&	45.67	&	0	\\
20	&	20	&	0.5	&	2	&	1	&	4	&	0.68	&	0.32	&	0.48	&	0.85	&	44.66	&	0.68	&	0.32	&	0.48	&	0.86	&	44.51	&	0	\\
20	&	20	&	1	&	1	&	2	&	2	&	0.38	&	0.82	&	0.82	&	0.21	&	42.61	&	0.38	&	0.82	&	0.82	&	0.22	&	42.67	&	0	\\
20	&	20	&	1	&	2	&	1	&	5	&	0.56	&	0.44	&	0.31	&	1.52	&	41.76	&	0.56	&	0.44	&	0.3	&	1.54	&	41.65	&	0	\\
20	&	20	&	2	&	1	&	2	&	3	&	0.31	&	0.98	&	0.73	&	0.37	&	42.41	&	0.3	&	0.98	&	0.73	&	0.38	&	42.43	&	0	\\
20	&	20	&	2	&	2	&	2	&	4	&	0.31	&	0.98	&	0.54	&	0.77	&	41.52	&	0.3	&	0.98	&	0.55	&	0.8	&	41.5	&	0	\\
   \hline
    \multicolumn{16}{c}{Fixed parameters: $C^{1}_{u}=20$, $C^{1}_d=40$,$C^{1}_a=C^{1}_a=0.25$,$C_w=0.5$.}\\
    \hline
    \end{tabular}}}
\end{center}
\end{table}

\section{Value of Priority versus age-threshold}\label{Append-3}
\begin{table}[H]
\begin{center}
{\footnotesize
\renewcommand{\arraystretch}{0.6}
\setlength{\tabcolsep}{0.4cm}
  \caption{Value of Priority vs age-threshold}\label{tb4}
  \scalebox{0.75}{%
    \begin{tabular}{cc|cc|ccccc|ccccc|c}
    \hline
    \multicolumn{4}{c|}{Cost parameters} & \multicolumn{5}{c|}{Optimal policy}      & \multicolumn{5}{c|}{Optimal policy} &    Difference\\
    \cline{1-4}
    \multicolumn{2}{c|}{Item 1} & \multicolumn{2}{c|}{Item 2}&  \multicolumn{5}{c|}{with priority}&  \multicolumn{5}{c|}{without priority}&in\\  
    $C_u$& $C_d$& $C_u$& $C_d$& $\tau_1$& $\tau_1$&$S_1$&$S_2$& $Z_1$&$\tau_1$&$\tau_2$&$S_1$&$S_2$&$Z_2$&  \%  \\
\hline
1	&	20	&	1	&	5	&	5	&	5	&	7	&	1	&	9.55	&	5	&	5	&	7	&	1	&	9.92	&	-4	\\
1	&	20	&	5	&	5	&	5	&	2	&	7	&	1	&	9.65	&	5	&	2	&	8	&	1	&	10.43	&	-8	\\
1	&	40	&	1	&	5	&	5	&	5	&	10	&	1	&	11.18	&	5	&	5	&	12	&	1	&	12.65	&	-13	\\
1	&	40	&	5	&	5	&	5	&	2	&	10	&	1	&	11.25	&	5	&	3	&	12	&	1	&	13.04	&	-16	\\
5	&	20	&	5	&	5	&	5	&	2	&	6	&	1	&	13.23	&	5	&	2	&	6	&	1	&	13.74	&	-4	\\
5	&	40	&	1	&	5	&	5	&	5	&	9	&	1	&	14.9	&	5	&	5	&	11	&	1	&	16.18	&	-9	\\
5	&	40	&	5	&	5	&	5	&	2	&	10	&	1	&	14.98	&	5	&	3	&	12	&	1	&	16.6	&	-11	\\
1	&	10	&	20	&	5	&	5	&	1	&	5	&	1	&	8.56	&	5	&	1	&	5	&	1	&	8.95	&	-5	\\
1	&	20	&	10	&	5	&	5	&	1	&	7	&	1	&	9.69	&	5	&	1	&	8	&	1	&	10.74	&	-11	\\
1	&	20	&	20	&	5	&	5	&	1	&	7	&	1	&	9.72	&	5	&	1	&	8	&	1	&	10.87	&	-12	\\
1	&	40	&	10	&	5	&	5	&	1	&	10	&	1	&	11.28	&	5	&	2	&	13	&	1	&	13.33	&	-18	\\
1	&	40	&	1	&	10	&	5	&	5	&	9	&	1	&	15.95	&	5	&	5	&	10	&	1	&	16.55	&	-4	\\
1	&	40	&	20	&	5	&	5	&	1	&	10	&	1	&	11.3	&	5	&	1	&	13	&	1	&	13.43	&	-19	\\
1	&	40	&	5	&	10	&	5	&	4	&	10	&	1	&	16.04	&	5	&	5	&	11	&	1	&	17.01	&	-6	\\
5	&	20	&	10	&	5	&	5	&	1	&	6	&	1	&	13.28	&	5	&	1	&	7	&	1	&	14.05	&	-6	\\
5	&	20	&	20	&	5	&	5	&	1	&	6	&	1	&	13.31	&	5	&	1	&	7	&	1	&	14.19	&	-7	\\
5	&	40	&	10	&	5	&	5	&	1	&	10	&	1	&	15.01	&	5	&	2	&	12	&	1	&	16.89	&	-13	\\
5	&	40	&	20	&	5	&	5	&	1	&	10	&	1	&	15.03	&	5	&	1	&	12	&	1	&	16.99	&	-13	\\
5	&	40	&	5	&	10	&	5	&	4	&	9	&	1	&	19.75	&	5	&	5	&	10	&	1	&	20.5	&	-4	\\
1	&	20	&	20	&	10	&	5	&	1	&	7	&	1	&	14.55	&	5	&	1	&	7	&	1	&	15.1	&	-4	\\
1	&	40	&	10	&	10	&	5	&	2	&	10	&	1	&	16.12	&	5	&	2	&	11	&	1	&	17.46	&	-8	\\
1	&	40	&	20	&	10	&	5	&	1	&	10	&	1	&	16.18	&	5	&	1	&	12	&	1	&	17.91	&	-11	\\
5	&	40	&	10	&	10	&	5	&	2	&	9	&	1	&	19.84	&	5	&	2	&	11	&	1	&	20.97	&	-6	\\
5	&	40	&	20	&	10	&	5	&	1	&	9	&	1	&	19.9	&	5	&	1	&	11	&	1	&	21.44	&	-8	\\
   \hline
    \end{tabular}}}
\end{center}
\end{table}

\begin{table}[H]
\begin{center}
{\footnotesize
\renewcommand{\arraystretch}{0.6}
\setlength{\tabcolsep}{0.25cm}
  \caption{Priority versus non priority}  \label{PrioVSNonPrio}
  \scalebox{0.75}{%
    \begin{tabular}{|c|l|lrrrr|rrrr|rrrr|rrrr|}
\cline{4-19}    \multicolumn{1}{r}{} & \multicolumn{1}{r}{} & \multicolumn{1}{c|}{} & \multicolumn{16}{c|}{Captial good 1 with high priority} \\
\cline{3-19}    \multicolumn{1}{r}{} &       & \multicolumn{1}{l|}{$C_u$} & \multicolumn{4}{l|}{1}        & \multicolumn{4}{l|}{5}        & \multicolumn{4}{l|}{10}       & \multicolumn{4}{l|}{20} \\
\cline{2-19}    \multicolumn{1}{c|}{} & {$C_u$}    & \multicolumn{1}{l|}{$C_d$} & \multicolumn{1}{l}{5} & \multicolumn{1}{l}{10} & \multicolumn{1}{l}{20} & \multicolumn{1}{l|}{40} & \multicolumn{1}{l}{5} & \multicolumn{1}{l}{10} & \multicolumn{1}{l}{20} & \multicolumn{1}{l|}{40} & \multicolumn{1}{l}{5} & \multicolumn{1}{l}{10} & \multicolumn{1}{l}{20} & \multicolumn{1}{l|}{40} & \multicolumn{1}{l}{5} & \multicolumn{1}{l}{10} & \multicolumn{1}{l}{20} & \multicolumn{1}{l|}{40} \\
    \hline
    \multirow{16}[8]{*}{\rotatebox[origin=c]{90}{Captial good 2 with low priority}} & \multirow{4}[2]{*}{1} &  \multicolumn{1}{l|}{5}     & 4   & -1  & -4  & -13 & 13  & 3   & 0   & -9  & 19  & 7   & -1  & -5  & 23  & 13  & 2   & -4 \\
          &       &  \multicolumn{1}{l|}{10}     &       & 6   & 0   & -4  &       & 7   & 2   & -1  &       & 16  & 3   & 0   &       & 24  & 8   & 0 \\
          &       &  \multicolumn{1}{l|}{20}      &       &       & 7   & 1   &       &       & 7   & 2   &       &       & 8   & 4   &       &       & 17  & 3 \\
          &       & \multicolumn{1}{l|}{40}    &       &       &       & 8   &       &       &       & 8   &       &       &       & 7   &       &       &       & 8 \\
\cline{2-19}          & \multirow{4}[2]{*}{5} & \multicolumn{1}{l|}{5}      &       &       &       &       & 10  & 1   & -4  & -10 & 15  & 4   & -2  & -7  & 17  & 9   & 0   & -5 \\
          &       & \multicolumn{1}{l|}{10}     &       &       &       &       &       & 7   & 0   & -4  &       & 14  & 2   & -2  &       & 21  & 7   & -1 \\
          &       & \multicolumn{1}{l|}{20}     &       &       &       &       &       &       & 6   & 1   &       &       & 7   & 2   &       &       & 15  & 3 \\
          &       & \multicolumn{1}{l|}{40}     &       &       &       &       &       &       &       & 7   &       &       &       & 7   &       &       &       & 7 \\
\cline{2-19}          & \multirow{4}[2]{*}{10} & \multicolumn{1}{l|}{5}      &       &       &       &       &       &       &       &       & 12  & 2   & -4  & -8  & 14  & 6   & -1  & -6 \\
          &       & \multicolumn{1}{l|}{10}     &       &       &       &       &       &       &       &       &       & 13  & 1   & -3  &       & 19  & 6   & -2 \\
          &       &\multicolumn{1}{l|}{20}     &       &       &       &       &       &       &       &       &       &       & 7   & 1   &       &       & 16  & 2 \\
          &       &\multicolumn{1}{l|}{40}    &       &       &       &       &       &       &       &       &       &       &       & 6   &       &       &       & 7 \\
\cline{2-19}          & \multirow{4}[2]{*}{20} & \multicolumn{1}{l|}{5}      &       &       &       &       &       &       &       &       &       &       &       &       & 10  & 3   & -3  & -6 \\
          &       &\multicolumn{1}{l|}{10}    &       &       &       &       &       &       &       &       &       &       &       &       &       & 16  & 3   & -4 \\
          &       & \multicolumn{1}{l|}{20}    &       &       &       &       &       &       &       &       &       &       &       &       &       &       & 16  & 1 \\
          &       & \multicolumn{1}{l|}{40}     &       &       &       &       &       &       &       &       &       &       &       &       &       &       &       & 7 \\
   \hline
    \end{tabular}}}
\end{center}
\end{table}

%\section*{Acknowledgements}$$
%The second author gratefully acknowledges the support of the National Research Fund of Luxembourg (AFR grant 12451704). The work of the third author is supported by the Data Science Flagship framework, a cooperation between Eindhoven University of Technology and Philips.

%% References


%\bibliographystyle{elsarticle-num}
%\bibliographystyle{elsarticle-num-names}
% \bibliographystyle{model1a-num-names}
% \bibliographystyle{model1b-num-names}
% \bibliographystyle{model1c-num-names}
%\bibliographystyle{model2-names}
% \bibliographystyle{model3a-num-names}
% \bibliographystyle{model3-num-names}
% \bibliographystyle{model4-names}

%\clearpage
\bibliographystyle{model5-names}
\bibliography{sample}
\end{document}